\documentclass[12pt]{article}
\usepackage{amsmath,amsfonts,amssymb,amsthm}
\usepackage{enumerate}

\newtheorem{proposition}{Proposition}[section]
\newtheorem{theorem}[proposition]{Theorem}
\newtheorem{corollary}[proposition]{Corollary}
\newtheorem{lemma}[proposition]{Lemma}

\theoremstyle{remark}
\newtheorem{remark}[proposition]{Remark}

\newtheorem{example}[proposition]{Example}

\newcommand{\nc}{\newcommand}
\nc{\I}{{\mathbf 1}}

\usepackage{color}

\nc{\bN}{{\mathbf N}}
\nc{\cB}{{\mathcal B}}
\nc{\cF}{{\mathcal F}}
\nc{\cP}{{\mathcal P}}
\nc{\cX}{{\mathcal X}}
\nc{\cS}{{\mathcal S}}
\nc{\cL}{{\mathcal L}}
\nc{\cN}{{\mathcal N}}
\nc{\R}{{\mathbb R}}
\nc{\N}{{\mathbb N}}
\nc{\Z}{{\mathbb Z}}
\nc{\BX}{{\mathbb X}}
\nc{\BY}{{\mathbb Y}}

\DeclareMathOperator{\dom}{dom}
\nc{\BP}{\mathbb{P}}
\nc{\BE}{\mathbb{E}}
\nc{\BQ}{\mathbb{Q}}
\nc{\bH}{\overline{H}}
\DeclareMathOperator{\BV}{{\mathbb Var}}

\newcommand{\Dirac}{\delta}
\newcommand{\deltaKS}{\boldsymbol{\delta}}

\numberwithin{equation}{section}

\begin{document} 

\renewcommand{\thefootnote}{\fnsymbol{footnote}}
\author{G. Last\footnotemark[1]\,, I. Molchanov\footnotemark[2]\, 
and M. Schulte\footnotemark[3]}
\footnotetext[1]{guenter.last@kit.edu, Karlsruhe Institute of
  Technology, Institute for Stochastics. 
	}
\footnotetext[2]{ilya.molchanov@stat.unibe.ch, University of Bern,
  Institute of Mathematical Statistics and Actuarial Science. 
	}
\footnotetext[3]{matthias.schulte@tuhh.de, Hamburg University of Technology,
  Institute of Mathematics. 
	}

\title{Normal approximation of Kabanov--Skorohod integrals\\ on Poisson spaces} 
\date{\today}
\maketitle
\begin{abstract}\noindent
  We consider the normal approximation of Kabanov--Skorohod integrals
  on a general Poisson space. Our bounds are for the Wasserstein and
  the Kolmogorov distance and involve only difference operators of the
  integrand of the Kabanov--Skorohod integral. The proofs rely on the
  Malliavin--Stein method and, in particular, on multiple applications
  of integration by parts formulae. As examples, we study some linear
  statistics of point processes that can be constructed by Poisson
  embeddings and functionals related to Pareto optimal points of a
  Poisson process.
\end{abstract}
\noindent
{\bf Keywords:} Kabanov--Skorohod integral, Poisson process, normal
approximation, Stein's method, Malliavin calculus

\vspace{0.1cm}
\noindent
{\bf AMS MSC 2020:} Primary: 60F05; secondary: 60G55, 60H05, 60H07

\section{Introduction}

Let $\eta$ be a {\em Poisson process} on a measurable space
$(\BX,\cX)$ with a $\sigma$-finite intensity measure $\lambda$,
defined on some probability space $(\Omega,\cF,\BP)$.  Formally,
$\eta$ is a {\em point process}, which is a random element of the
space $\bN$ of all $\sigma$-finite measures on $\BX$ with values in
$\mathbb{N}_0\cup\{\infty\}$, equipped with the smallest
$\sigma$-field $\cN$ making the mappings $\mu\mapsto\mu(B)$ measurable
for each $B\in\cX$.  The Poisson process $\eta$ is \emph{completely
  independent}, that is, $\eta(B_1),\hdots,\eta(B_n)$ are independent
for pairwise disjoint $B_1,\hdots,B_n\in\mathcal{X}$,
$n\in\mathbb{N}$, and $\eta(B)$ has for each $B\in\cX$ a Poisson
distribution with parameter $\lambda(B)$, see e.g.\
\cite{Kallenberg02,LastPenrose17}.

Let $G\colon \bN\times\BX\to\R$ be a measurable function which is
square integrable with respect to $\BP_\eta\otimes\lambda$, where
$\BP_\eta:=\BP(\eta\in\cdot)$ denotes the distribution of $\eta$. In
this paper, we study the Kabanov--Skorohod integral (short:
KS-integral) of $G$ defined as a Malliavin operator. If $G$ is in the
domain of the KS-integral and integrable with respect to
$\BP_\eta\otimes\lambda$, its KS-integral is pathwise given by
\begin{align}\label{e1.1}
  \deltaKS(G)=\int G_x(\eta-\Dirac_x)\,\eta(dx)-\int G_x(\eta)\,\lambda(dx),
\end{align}
where $\Dirac_x$ stands for the Dirac measure at $x\in\BX$, 
see e.g.\ \cite[Theorem~6]{Last16}. In this case, the Mecke formula
immediately yields that $\BE \deltaKS(G)=0$. We refer to \cite{Last16}
for an introduction to stochastic calculus on a general Poisson
space.

The pathwise representation \eqref{e1.1} of the KS-integral consists
of two terms. The first term is the sum of the values
$G_x(\eta-\Dirac_x)$ over the points of $\eta$. Such sums have been
intensively studied. The state of the art of limit theorems for such
sums is presented in \cite{lac:sch:yuk19}, based on the idea of
stabilisation. The stabilisation property means that the functional
$G_x(\eta-\Dirac_x)$ depends only on points of $\eta$ within some
finite random distance from $x$, with conditions imposed on the
distribution of such a distance. As in \cite{lac:sch:yuk19}, we use recent developments
of the Malliavin--Stein technique for Poisson processes, first
elaborated in \cite{PSTU10} and then extended in
\cite{eich:tha16,LPY20,las:pec:sch16,schulte16}.

In all above mentioned works, the sums over Poisson processes are
centred by subtracting the expectation, which is
\begin{displaymath}
  \BE \int G_x(\eta-\Dirac_x)\,\eta(dx)
  = \int \BE G_x(\eta) \,\lambda(dx).
\end{displaymath}
In contrast, the centring involved in the pathwise construction of the
KS-integral in \eqref{e1.1} is random. As shown in \cite{LM21},
KS-integrals naturally appear in the construction of unbiased estimators
derived from Poisson hull operators.

In this paper we derive bounds for the Wasserstein and the Kolmogorov
distance between $\deltaKS(G)$ and a standard normal random variable.
Limit theorems for compensated stochastic Poisson integrals in the Wasserstein distance have
been studied in several papers by N.~Privault, assuming that $\BX$ is
the Euclidean space $\R^d$ with separate treatments of the cases $d=1$
in \cite{priv19} and $d\geq2$ in \cite{priv18}. In \cite{priv19} the
integrand is assumed to be adapted and in \cite{priv18} it is assumed
to be predictable and to have bounded support. In particular, the stochastic integral
coincides in both cases with the KS-integral. Under these assumptions,
the tools, based on derivation operators and Edgeworth-type
expansions, have resulted in bounds involving integrals of the third
power of $G$ and differential operators applied to $G$. In comparison,
our results apply to a general state space, are not restricted to
predictable (or adapted) integrands, and do not assume the support of
the integrand to be bounded in any sense.  Furthermore, our bounds are
given in terms of difference operators directly applied to the
integrand $G$, and are derived for both the Wasserstein and the
Kolmogorov distance. However, our bounds contain the integral of the
absolute value of $G$ to power 3, which may be larger than the
corresponding term in \cite{priv18}.  Our results are used in
\cite{LM21} to derive quantitative central limit theorems.

Let us compare our proof strategy with the standard approach for
  the normal approximation of Poisson functionals via the
  Malliavin--Stein method, which goes back to \cite{PSTU10} and is
  also employed in \cite{eich:tha16,las:pec:sch16,LPY20,schulte16}. To this end, we
  omit all technical assumptions and definitions (some will be given
  later). Let $F$ be a Poisson functional (a measurable function of $\eta$) and let $f$ be the solution
  of the associated Stein equation.
The identity $\deltaKS D=-L$, where $D$ is
  the difference operator and $L$ is the Ornstein--Uhlenbeck generator
  with its inverse $L^{-1}$, and integration by parts lead to
$$
\BE F f(F) = \BE \deltaKS(-DL^{-1}F) f(F) = \BE \int (-D_xL^{-1}F) D_xf(F) \, \lambda(dx).
$$
This step comes for the price of the term $D_xL^{-1}F$, which is often
difficult to evaluate and whose treatment is one of the main
achievements of \cite{las:pec:sch16}. For the special case
$F=\deltaKS(G)$ the identity $\deltaKS D=-L$ is not required. Instead,
an immediate integration by parts yields that 
\begin{equation}\label{eqn:inegration_by_parts_approach}
\BE F f(F) = \BE \deltaKS(G) f(\deltaKS(G)) = \BE \int G_x D_xf(\deltaKS(G)) \, \lambda(dx),
\end{equation}
avoiding the inverse Ornstein--Uhlenbeck generator.
We treat the KS-integrals that arise from the Taylor
expansion of $D_xf(\deltaKS(G))$ also by integration by parts,
so that our final bounds only involve $G$ and its difference operators
but no KS-integrals. This is a difference to \cite{Torrisi17}, 
where the argument in
\eqref{eqn:inegration_by_parts_approach} is used but no further
integration by parts.

Even though our proofs differ from previous
  works, one may wonder whether existing Malliavin-Stein
  bounds can be applied to $\deltaKS(G)$. As they do not involve the
  inverse Ornstein-Uhlenbeck generator, the results from
  \cite{las:pec:sch16} seem to be the best ones for the off-the-shelf
  use. They require only moments of the first and the second
  difference operator of the Poisson functional $F$, which one might
  also encounter when evaluating the bounds from
  \cite{eich:tha16,LPY20,PSTU10,schulte16}. In our case, this means
  that one has to control moments like $\BE\big[ G_x^4 \big]$,
  $\BE\big[ \deltaKS(D_xG)^4 \big]$ and
    $\BE\big[ \deltaKS(D_{x,y}^2G)^4 \big]$ for $x,y\in\BX$. Since we
  aim for bounds in terms of $G$ and its difference operators,
  one has to remove the KS-integrals. This can be achieved by fourfold
  integration by parts, but would lead to normal approximation bounds
  that are more involved than in the current paper and include
    even iterated integrals with roots of the inner integrals. We
  expect these results to yield the same rates of convergence as our
  approach but under stronger integrability assumptions. Instead, our
  approach is direct and leads to much simpler calculations.
In particular, it does not require the computation of expressions involving powers of
  KS-integrals apart from second moments.

Section~\ref{sec:main-results} presents our main results, which are
proved in Sections \ref{sec:Wasserstein} and \ref{sec:Kolmogorov}
separately for the Wasserstein and Kolmogorov distances, after
recalling necessary results and constructions from stochastic calculus
on Poisson spaces in Section~\ref{secpre}. We conclude with two
examples in Sections \ref{sec:Poisson_embedding} and
\ref{sec:integrals-functions} concerning some linear statistics of
point processes constructed via Poisson embeddings and Pareto optimal
points. 

\section{Main results}
\label{sec:main-results}

To state our results we need to introduce some notation. The
\emph{Wasserstein distance} between the laws of two integrable random variables
$X$ and $Y$ is defined by
\begin{align*}
  d_W(X,Y):=\sup_{h\in {\operatorname{\mathbf{Lip}}}(1)}
  \big|\BE[h(X)]-\BE[h(Y)]\big|,
\end{align*}
where ${\operatorname{\mathbf{Lip}}}(1)$ denotes the space of all
Lipschitz functions $h\colon\R\to\R$ with a Lipschitz constant at most
one.  The \emph{Kolmogorov distance} between the laws
of $X$ and $Y$ is given by
\begin{displaymath}
  d_K(X,Y):=\sup_{t \in \R} |\BP(X\le t)-\BP(Y\le t)|.
\end{displaymath}

Given a function $f\colon\bN\to\R$ and $x\in \BX$, the function
$D_xf\colon\bN\to\R$ is defined by
\begin{align}\label{def:D}
  D_xf(\mu):=f(\mu+\Dirac_x)-f(\mu),\quad \mu\in \bN.
\end{align} 
Then $D_x$ is known as the \emph{difference operator}. Iterating its
definition yields, for given $x,z,w\in \BX$, the second difference
operator $D^2_{x,z}$ and the third difference operator $D^3_{x,z,w}$ which can again be applied to functions $f$ as
above. For a function $G\colon \bN\times\BX\to\R$ (which maps
$(\mu,y)$ to $G_y(\mu)$) and $x,z,w\in\BX$ we let $D_x$, $D^2_{x,z}$ and $D^3_{x,z,w}$
act on $G_y(\cdot)$ so that it makes sense to talk about $D_xG_y(\mu)$, $D^2_{x,z}G_y(\mu)$
and $D^3_{x,z,w}G_y(\mu)$. Throughout the paper, 
we write shortly $G_y$ for $G_y(\eta)$ and similarly for difference operators.

We shall require the following integrability assumptions:
\begin{align}\label{eH2}
&\BE \int G^2_y\,\lambda(dy)<\infty,\\
\label{eH3}
&\BE \int (D_xG_y)^2\,\lambda^2(d(x,y))<\infty,\\
\label{eH4}
&\BE \int (D^2_{z,x} G_y)^2\,\lambda^3(d(x,y,z))<\infty,\\
\label{eH5}
&\BE \int (D^3_{w,z,x} G_y)^2\,\lambda^3(d(w,y,z))<\infty,\quad \lambda\text{-a.e.\ $x$}.
\end{align}
If \eqref{eH2} and \eqref{eH3} hold, 
it follows from \cite[Proposition~2.3]{las:pec:sch16} that the KS--integral $\deltaKS(G)$ of $G$ 
is defined as a Malliavin operator and satisfies
\begin{align}\label{evariance}
\BV \deltaKS(G) = \BE \int G_x^2 \, \lambda(dx) + \BE \int D_xG_y D_yG_x \, \lambda^2(d(x,y)).
\end{align}
In order to deal with the Kolmogorov distance, we also need to assume that
\begin{align}
& \BE \int |D_xG_y G_x|\,\lambda^2(d(x,y))<\infty \label{eH_mixed},\\
& \BE \int \big(D_z(G_x|G_x|)\big)^2 \, \lambda^2(d(x,z))<\infty \label{eH_mixed_2},\\
& \BE \int \bigg( \int D_z(D_xG_y D_y|G_x|) \, \lambda(dy) \bigg)^2 \, \lambda^2(d(x,z))<\infty \label{eH_mixed_3}
.
\end{align}

The following main result on the normal approximation of $\deltaKS(G)$
involves only the integrand $G$ and its first, second and third order
difference operators.  Throughout the paper we let $N$ denote a
standard normal random variable. Define
and denote
\begin{align*}
  T_1&:=\left(\BE \int \left(\int D_y(G_x^2) \,
      \lambda(dx)\right)^2\,  \lambda(dy)\right)^{1/2}, \allowdisplaybreaks\\
  T_2&:=\left(\BE \int \left( \int D_z(D_xG_y D_yG_x) \,
       \,\lambda^2(d(x,y))\right)^2 \,\lambda(dz)\right)^{1/2}, \allowdisplaybreaks \\
  T_3 & := \BE \int |G_x|^3 \, \lambda(dx), \allowdisplaybreaks \\
  T_4 & := \BE \int \Big( 3 \big|D_xG_y D_yG_x G_x\big| + \big|D_xG_y (D_yG_x)^2\big|
        + 2G_x^2 \big|D_xG_y\big| \\
     & \qquad \qquad + \big| (G_x + D_yG_x)  D_xG_y
       \big| \big(2|G_y|+|D_xG_y+D_yG_x|\big) \Big)
       \, \lambda^2(d(x,y)), \allowdisplaybreaks \\
  T_5 & := \BE \int \bigg( 2\big(|D_yG_z|+|D^2_{x,y}G_z|\big)
        \Big(|D_z\big((G_x + D_yG_x) D_xG_y\big)| + 2|(G_x + D_yG_x) D_xG_y| \Big) \\
     & \qquad \qquad +   |D_xG_z| \Big(|D_z\big(D_yG_x D_xG_y\big)|
       + 2|D_yG_x D_xG_y| \Big) \bigg) \, \lambda^3(d(x,y,z)), \allowdisplaybreaks \\
  T_6 & := \bigg(\BE \int \bigg( \int (G_x+D_yG_x) D_xG_y
        \, \lambda(dx)\bigg)^2\, \lambda(dy) \\
     & \qquad \qquad +\BE \int \bigg( \int D_z\big((G_x+D_yG_x)
       D_xG_y\big) \, \lambda(dx)\bigg)^2 \lambda^2(d(y,z)) \bigg)^{1/2}, \allowdisplaybreaks \\
  T_7 & := \bigg( \BE \int G_x^4 \, \lambda(dx)
        + \BE \int D_x(G_y|G_y|) D_y(G_x|G_x|) \, \lambda^2(d(x,y)) \bigg)^{1/2}, \allowdisplaybreaks \\
  T_8 & := \bigg( \BE \int \bigg( \int D_xG_y D_y|G_x|
        \, \lambda(dy) \bigg)^2\, \lambda(dx) \\
     & \qquad \qquad + \BE \int D_x\bigg( \int D_zG_y D_y|G_z|
       \, \lambda(dy) \bigg)  D_z\bigg( \int D_xG_y D_y|G_x|
       \lambda(dy) \bigg) \, \lambda^2(d(x,z))\bigg)^{1/2}, \allowdisplaybreaks \\
  T_9 & := \bigg( 3\BE \int (D_xG_y)^2 \big(D_y|G_x|+|G_x|\big)^2
        \, \lambda^2(d(x,y))  \\
     & \qquad \qquad + 3 \BE \int \Big( D_z\big( D_xG_y (D_y|G_x|+|G_x|)
       \big) \Big)^2 \, \lambda^3(d(x,y,z))  \\
     & \qquad \qquad + 2 \BE \int \Big( D^2_{z,w}\big( D_xG_y
       (D_y|G_x|+|G_x|)\big) \Big)^2 \, \lambda^4(d(x,y,z,w)) \bigg)^{1/2}. 
\end{align*}

\begin{theorem}\label{tmain}
Suppose that $G\colon\bN\times\BX\to\R $ satisfies the
assumptions \eqref{eH2}, \eqref{eH3}, \eqref{eH4} and \eqref{eH5}. Assume also that
  $\BE \deltaKS(G)^2=1$.  Then
  \begin{align}
    \label{eW}
    d_W(\deltaKS(G),N)\le T_1+T_2+T_3+T_4+T_5.
  \end{align}
If, additionally, \eqref{eH_mixed}, \eqref{eH_mixed_2} and \eqref{eH_mixed_3} are satisfied, then 
  \begin{align}
    \label{eK}
    d_K(\deltaKS(G),N)\le T_1+T_2+T_6+2(T_7+T_8+T_9).
  \end{align}
\end{theorem}

We say that the functional $G$ satisfies the \emph{cyclic condition
of order two} if
\begin{equation}
  \label{eq:3}
  D_xG_yD_yG_x=0 \quad \text{a.s.} \quad \text{for} \quad  \lambda^2\text{-a.e.} \quad  (x,y)\in\BX^2,
\end{equation}
see \cite{Privault12}, where such
conditions were used to simplify moment formulae for the
KS-integral. Note that \eqref{eq:3} always holds if the functional
$G$ is predictable, that is, the carrier space is equipped with a
strict partial order $\prec$ and $G_y(\eta)$ depends only on $\eta$ restricted
to $\{x\in\BX:x\prec y\}$. If \eqref{eq:3} holds, then also 
$$
D_x|G_y| D_y|G_x| = 0 \quad \text{and} \quad D_xG_y D_y|G_x| = 0 \quad \text{a.s.} \quad \text{for} \quad \lambda^2\text{-a.e.} \quad (x,y)\in\BX^2,
$$
since
\begin{align*}
0  \leq\big| D_x|G_y| D_y|G_x| \big| =  \big| D_x|G_y|\big|  \big|D_y|G_x| \big| \leq |D_xG_y|  \big|D_y|G_x| \big|  \leq |D_xG_y| |D_yG_x| = 0.
\end{align*}
In view of this, under the cyclic condition, the bounds from
Theorem~\ref{tmain} simplify as follows.

\begin{corollary}\label{cor:cyclic}
  Assume that the cyclic condition \eqref{eq:3} holds,
  and the assumptions of Theorem~\ref{tmain} are maintained. Then the
  bounds \eqref{eW} and \eqref{eK} hold with $T_2=T_8=0$, and
  \begin{align*}
    T_4&=\BE \int  \Big(2G_x^2 |D_xG_y|
         + \big| G_x  D_xG_y\big| \big(2|G_y|+|D_xG_y|\big)\Big)
         \, \lambda^2(d(x,y)),\\
    T_5&=\BE \int 2\Big(|D_yG_z|+|D^2_{x,y}G_z|\Big)
         \Big(\big|D_z(G_x D_xG_y)\big| + 2 |G_xD_xG_y| \Big)
         \, \lambda^3(d(x,y,z)),\\
    T_6&=\bigg( \BE \int \bigg( \int G_x D_xG_y \, \lambda(dx)\bigg)^2 \, \lambda(dy) 
         + \BE \int \bigg( \int D_z\big(G_x D_xG_y\big)
         \, \lambda(dx)\bigg)^2 \, \lambda^2(d(y,z)) \bigg)^{1/2},\\
    T_7&= \bigg( \BE \int G_x^4 \, \lambda(dx) \bigg)^{1/2},\\
    T_9&= \bigg( 3\BE \int (D_xG_y)^2 G_x^2 \, \lambda^2(d(x,y)) + 3 \BE \int \big( D_z( D_xG_y |G_x| )  \big)^2 \, \lambda^3(d(x,y,z)) \\
        & \qquad \qquad + 2 \BE \int \big( D^2_{z,w}(D_xG_y |G_x|) \big)^2 \, \lambda^4(d(x,y,z,w)) \bigg)^{1/2}.
  \end{align*}
\end{corollary}

\begin{remark}
  Assuming that $G_x(\eta)\equiv f(x)$ does not depend on $\eta$ and
  that $f\in L^2(\lambda)$, $\deltaKS(G)$ is the first Wiener--It\^o
  integral $I_1(f)$ of $f$ (see e.g.\ \cite[Chapter
  12]{LastPenrose17}).  In this case, Theorem \ref{tmain} yields the
  classical Stein bounds for the Wasserstein and the Kolmogorov
  distance,
  \begin{align*}
    d_W(I_1(f),N)
    \le \int |f(x)|^3 \, \lambda(dx)
  \end{align*}
  and
  \begin{align*}
    d_K(I_1(f),N)
    \le 2\bigg(\int f(x)^4 \, \lambda(dx) \bigg)^{1/2},
  \end{align*}
  see e.g.\ \cite[Corollary~3.4]{PSTU10} and
  \cite[Example~1.3]{las:pec:sch16}.
\end{remark}

\begin{remark}
  The paper \cite{Torrisi17} studies normal and Poisson
    approximation of innovations of general point processes with
    Papangelou conditional intensities, which include KS-integrals on
    the Poisson space. More precisely, Theorem~3.1 and Corollary~3.2
    in \cite{Torrisi17} (with $\pi=1$ there) provide bounds on the
    Wasserstein distance between a KS-integral and a standard normal
    random variable. In contrast to our main results, the
      bound in Theorem~3.1 from
    \cite{Torrisi17} contains still KS-integrals as integration by
    parts is employed only once. Proceeding there with further
    integrations by parts might be challenging since one of the
    KS-integrals is within an absolute value. The bound on the
    Wasserstein distance presented in
    Theorem~3.1 is evaluated in Corollary~3.2, but the resulting bound
    might not always behave as desired 
for a limit theorem. The first term on the
    right-hand side can be bounded from below by
$$
\bigg|1- \BE \int G_x^2 \, \lambda(dx)\bigg|,
$$ 
which does not become small if the KS-integral has variance one and
the second term in \eqref{evariance} has a non-vanishing contribution
(see Example~6.5 for such a situation). The third term contains only a
product of two factors, which could be not sufficient
if one rescales by the
standard deviation of the KS-integral (see e.g.\ the situation
discussed in Remarks \ref{rpredict} and
\ref{rpredict3} under the additional assumptions that $u$ is constant and $\varphi$ is translation
invariant in its first argument).
\end{remark}

\begin{remark}
  In view of the works \cite{PZ,SY}, we expect that our results can be
  extended to the multivariate normal approximation of vectors of
  KS-integrals for distances based on smooth test functions and
  for the so-called $d_{\mathrm{convex}}$-distance under suitable assumptions.
\end{remark}

\section{Preliminaries}\label{secpre}

In this section we provide some basic properties of the difference
operator $D$ and the KS-integral $\deltaKS$. First of all, we recall
from \cite{Last16} the definitions of $D$ and $\deltaKS$ as {\em
  Malliavin operators}.  These definitions are based on $n$-th order
Wiener--It\^o integrals $I_n$, $n\in\N$; see also
\cite[Chapter~12]{LastPenrose17}. For symmetric functions
$f\in L^2(\lambda^n)$ and $g\in L^2(\lambda^m)$ with
$n,m\in\N$ we have
\begin{equation}\label{eqn:Isometry}
  \BE I_n(f) I_m(g) = \mathbf{1}\{n=m\} n! \int f(x) g(x) \, \lambda^n(dx).
\end{equation}
We use the convention $I_0(c)=c$ for $c\in\R$. Any
$H\in L^2(\BP_\eta)$ admits a \emph{chaos expansion}
\begin{align}\label{echaos}
  H=\sum^\infty_{n=0} I_n(h_n),
\end{align}
where we recall our (somewhat sloppy) convention $H\equiv H(\eta)$,
and where $h_0=\BE H$ and the $h_n$, $n\in\N$, are symmetric elements of
$L^2(\lambda^n)$. Here and in the following, we mean by series of
Wiener--It\^o integrals their $L^2$-limit, whence all identities
involving such sums hold almost surely. Then $H$ is in the 
{\em domain $\dom D$ of the difference operator $D$} 
(in the sense of a Malliavin operator)  if
\begin{align*}
  \sum^\infty_{n=1}nn!\int h_n(x_1,\ldots,x_n)^2\,\lambda^n(d(x_1,\ldots,x_n))<\infty.
\end{align*}
In this case one has
\begin{align*}
  D_xH=\sum^\infty_{n=1}nI_{n-1}(h_n(x,\cdot)),\quad \lambda\text{-a.e.} \ x\in \BX,
\end{align*}
see \cite[Theorem~3]{Last16}, i.e., the pathwise defined difference
operator from \eqref{def:D} can be represented in terms of the chaos
expansion \eqref{echaos}.  For $H\in L^2(\BP_\eta)$ the relations
$H\in \dom D$ and
\begin{displaymath}
  \BE \int (D_xH)^2\,\lambda(dx)<\infty
\end{displaymath}
are equivalent; see \cite[Eq.~(48)]{Last16}.
The (pathwise defined) difference operator satisfies
the product rule   
\begin{equation} \label{eq:1}
  D_x(HH')=(D_xH)(H+D_xH')+HD_xH',\quad x\in\BX,
\end{equation}
for measurable $H,H'\colon\bN\to\R$.

Now  let $G\colon \bN\times\BX\to\R$ be a measurable function
such that $G_x\equiv G(\cdot,x)\in L^2(\BP_\eta)$ for $\lambda$-a.e.\ $x$.
Then there exist measurable functions
$g_n\colon \BX^{n+1}\to\R$, $n\in\N_0$, such that 
\begin{align}\label{2.4}
  G_x=\sum_{n=0}^\infty I_n(g_n(x,\cdot)), \quad 
  \lambda\text{-a.e.\ $x\in\BX$.}
\end{align}
One says that $G$ is in the domain $\dom\deltaKS$ of the KS-integral $\deltaKS$ if
\begin{equation*}
  \sum_{n=0}^\infty (n+1)! \int \tilde{g}_n(\mathbf{x})^2\,\lambda^{n+1}(d\mathbf{x})<\infty,
\end{equation*}
where $\tilde{g}_n\colon \BX^{n+1}\to\R$ is the symmetrisation of $g_n$.
In this case the KS-integral of $G$ is defined by
\begin{align}\label{eKSI}
  \deltaKS(G):=\sum_{n=0}^\infty I_{n+1}(\tilde{g}_n).
\end{align}
We have $\BE\deltaKS(G)=0$. If $G\in\dom\deltaKS\cap L^1(\BP_\eta\otimes\lambda)$,
then $\deltaKS(G)$ is indeed given by the pathwise formula \eqref{e1.1};
see \cite[Theorem~6]{Last16}. If $G\in L^2(\BP_\eta\otimes\lambda)$, which is \eqref{eH2},
and if \eqref{eH3} holds,
then $G\in\dom\deltaKS$ and
\begin{equation} \label{eq:4}
\BE\deltaKS(G)^2
=\BE\int G_x^2\,\lambda(dx)+\BE\int D_xG_y D_yG_x \,\lambda^2(d(x,y)), 
\end{equation}
see \cite[Proposition~2.3]{las:pec:sch16} or
\cite[Theorem~5]{Last16}. Thus, the assumptions \eqref{eH2} and
\eqref{eH3} on $G$ in Theorem \ref{tmain} are sufficient to guarantee
that $G\in\operatorname{dom}\deltaKS$.

For $H\in\dom D$ and $G\in\dom \deltaKS$ we have the
important \emph{integration by parts} formula
\begin{equation}\label{e2.2}
  \BE H \deltaKS(G)=\BE\int G_x D_xH \,\lambda(dx);
\end{equation}
see e.g.\ \cite[Theorem~4]{Last16}.  Unfortunately, the assumption
$H\in\dom D$ is often not easy to check, and the sufficient conditions
given above lead to rather strong integrability assumptions. Instead
we shall often use the following two results.

\begin{lemma}\label{l3.1}
  Suppose that $G$ satisfies \eqref{eH2} and \eqref{eH3}, and let
  $H\in L^2(\BP_\eta)$ be such that $D_xH\in L^2(\BP_\eta)$ for
  $\lambda$-a.e.\ $x$. Then
  \begin{align}\label{efin}
    \int |\BE D_xH G_x|\,\lambda(dx)<\infty
  \end{align}
  and \eqref{e2.2} holds.
\end{lemma}
\begin{proof}
  The proof is essentially that of Lemma 2.3 in \cite{schulte16}. For
  the convenience of the reader we provide the main arguments.  Since
  $H\in L^2(\BP_\eta)$, we can represent $H$ as in
  \eqref{echaos}. Similarly, we can write
  \begin{align*}
    D_xH=\sum^\infty_{n=0} I_n(h'_n(x,\cdot)), \quad \lambda\text{-a.e.} \ x,
  \end{align*}
  where the measurable functions $h'_n\colon \BX^{n+1}\to \R$ are in
  the last $n$ coordinates symmetric and square integrable with
  respect to $\lambda^n$. In fact, it follows from
  \cite[Theorem~18.10]{LastPenrose17} that we can choose
  \begin{align*}
    h'_n(x,\mathbf{x})=(n+1)h_{n+1}(x,\mathbf{x}).
  \end{align*}
  Combining this with \eqref{2.4} and \eqref{eqn:Isometry}, we obtain
  \begin{align*}
    \BE D_xH G_x=\sum^\infty_{n=0}(n+1)!
    \int h_{n+1}(x,\mathbf{x})g_{n}(x,\mathbf{x})\,\lambda^n(d\mathbf{x})
  \end{align*}
  for $\lambda$-a.e.\ $x$. The Cauchy--Schwarz inequality (applied
  twice) yields
  \begin{align*}
    \int &|\BE D_xH G_x|\,\lambda(dx)\\
         &\le \Bigg(\sum^\infty_{n=0}(n+1)!\int
           h_{n+1}(\mathbf{x})^2\,\lambda^{n+1}(d\mathbf{x})\Bigg)^{1/2}
           \Bigg(\sum^\infty_{n=0}(n+1)!\int
           g_{n}(\mathbf{x})^2\,\lambda^{n+1}(d\mathbf{x})\Bigg)^{1/2}.
  \end{align*}
  Since $\BE H^2<\infty$, the first factor on the above
  right-hand side is finite.  By assumption \eqref{eH3}, the second
  factor is finite as well; see the proof of
  \cite[Theorem~5]{Last16}. Hence \eqref{efin} holds.  The remainder
  of the proof is as in \cite{schulte16}.
\end{proof}

\begin{lemma}\label{l3.2}
  Suppose that $G$ satisfies \eqref{eH2} and \eqref{eH3}, and let
  $H\colon \bN\to\R$ be a measurable function satisfying
  \begin{align}\label{eGH}
    \BE |H \deltaKS(G)|<\infty.
  \end{align}
  Then
  \begin{align}\label{epartialinequality}
    |\BE H \deltaKS(G)|\le \BE \int |D_xH G_x|\,\lambda(dx).
  \end{align}
\end{lemma}
\begin{proof} 
If $H$ is bounded,
then \eqref{epartialinequality}  follows from Lemma \ref{l3.1}.
In the general case we set $H_r:=\max\{\min\{H,r\}),-r\}$ for $r>0$. Then 
\eqref{epartialinequality} holds with $H_r$ instead of $H$. Hence, the observation that $|D_xH_r|\leq |D_xH|$ for $x\in\BX$ (see \cite[Exercise~18.4]{LastPenrose17}) yields that
\begin{align*}
|\BE H_r \deltaKS(G)|\le \BE \int |D_xH G_x|\,\lambda(dx).
\end{align*}
By \eqref{eGH} we can conclude the assertion from dominated convergence.
\end{proof}

We often need the following (basically) well-known
commutation rule for the KS-integral. For the pathwise defined
version \eqref{e1.1} this rule follows (under suitable integrability
assumptions) by direct calculation.

\begin{lemma}\label{l3.3}
  Suppose that $G$ satisfies \eqref{eH2}, \eqref{eH3} and
  \eqref{eH4}. Then $\deltaKS(G)\in\dom D$ and $D_xG\in\dom\deltaKS$ for
  $\lambda$-a.e.\ $x$ as well as
  \begin{equation} \label{eq:delta1}
    D_x\deltaKS(G)=G_x+\deltaKS(D_xG)\quad \text{a.s.},\, \lambda\text{-a.e.\ $x\in\BX$}.
  \end{equation}
\end{lemma}
\begin{proof}
  We have already noticed at 
\eqref{eq:4} that \eqref{eH2} and \eqref{eH3} imply $G\in\dom\deltaKS$.
Next we show that  $\deltaKS(G)\in\dom D$. 
Assumptions \eqref{eH2} and \eqref{eH3}
ensure that $G_x\in\dom D$ for $\lambda$-a.e.\ $x$.
Representing $G$ as in 
\eqref{2.4} and using \cite[Theorem~3]{Last16} twice, we can write 
\begin{align*}
  D^2_{y,z}G_x=\sum^\infty_{n=0}(n+2)(n+1) I_{n}(g_{n+2}(x,y,z,\cdot)), \quad 
\lambda^2\text{-a.e.\ $(y,z)\in\BX^2$}.
\end{align*}
By the $L^2$-convergence of the right-hand side and
\eqref{eqn:Isometry}, we obtain
\begin{align*}
  \BE &\int (D^2_{y,z} G_x)^2\,\lambda^3(d(x,y,z))\\
      &=\sum^\infty_{n=0}(n+2)^2(n+1)^2n!
        \iint g_{n+2}(x,y,z,\mathbf{x})^2\,\lambda^n(d \mathbf{x})
        \,\lambda^3(d(x,y,z))\\
      &=\sum^\infty_{n=0}(n+2)(n+1)(n+2)!
        \int g_{n+2}(\mathbf{x})^2\,\lambda^{n+3}(d \mathbf{x}).
\end{align*}
By assumption \eqref{eH4} this is finite, which is equivalent to 
\begin{align*}
\sum^\infty_{n=2}n (n-1) n! \int g_{n}(\mathbf{x})^2\,\lambda^{n+1}(d \mathbf{x})<\infty.
\end{align*}
In view of \eqref{eKSI} and the inequalities
\begin{align*}
\int \tilde{g}_{n}(\mathbf{x})^2\,\lambda^{n+1}(d \mathbf{x})
\le \int g_{n}(\mathbf{x})^2\,\lambda^{n+1}(d \mathbf{x})
\end{align*}
(a consequence of Jensen's inequality), this yields that $\deltaKS(G)\in\dom D$.

Let $G'$ be another measurable function satisfying \eqref{eH2} and
\eqref{eH3}. It follows from \eqref{eq:4} and the polarisation identity that
\begin{align} \label{e4.34}
\BE\deltaKS(G)\deltaKS(G')
=\BE\int G_xG_x'\,\lambda(dx)+\BE\int D_xG_y D_yG'_x \,\lambda^2(d(x,y)). 
\end{align}
The integration by parts formula \eqref{e2.2} yields that
\begin{align*}
\BE \deltaKS(G)\deltaKS(G')=\BE \int G'_x D_x\deltaKS(G)\,\lambda(dx).
\end{align*}
Assumptions \eqref{eH3} and \eqref{eH4} show that $D_xG\in\dom\deltaKS$
for $\lambda$-almost all $x$ and that $\deltaKS(D_\cdot G)$
  belongs to $L^2(\BP\otimes\lambda)$ (see \eqref{eq:4} and the
  discussion before it).  Therefore, we obtain from Fubini's
  theorem and integration by parts that
\begin{align*}
\BE\iint D_xG_y D_yG'_x \,\lambda(dy)\,\lambda(dx)
=\BE \int G'_x \deltaKS(D_xG)\,\lambda(dx),
\end{align*}
where we could apply Fubini's theorem on the left-hand side due
to \eqref{eH3} and on the right-hand side by the Cauchy--Schwarz
  inequality and the square integrability of $G'$ and
$\deltaKS(D_\cdot G)$.  Inserting these two results into
\eqref{e4.34} yields
\begin{align*}
\BE \int G'_x D_x\deltaKS(G)\,\lambda(dx)
=\BE\int G'_x G_x \,\lambda(dx)
+\BE \int G'_x \deltaKS(D_xG)\,\lambda(dx).
\end{align*}
Since the class of functions $G'$ with the required
properties is dense in $L^2(\BP_\eta\otimes\lambda)$
(see e.g.\ the proof of \cite[Theorem~5]{Last16}),
we conclude the asserted formula \eqref{eq:delta1}.
\end{proof}

\section{Proof for the Wasserstein distance in Theorem \ref{tmain}}\label{sec:Wasserstein}

Our proof is similar to the proofs of Theorems 1.1 and 1.2 in
\cite{las:pec:sch16} and relies on the ideas already present in
\cite{PSTU10}.  The first step is to recall Stein's method.  Let
$\mathbf{C}_{1,2}$ be the set of all twice continuously differentiable
functions $g\colon\R\rightarrow \R$ whose first derivative is bounded
in absolute value by $1$ and the second derivative by $2$. 
Then we have for an integrable random variable $X$ that
\begin{displaymath}
  d_W(X,N)\le\sup_{g\in \mathbf{C}_{1,2}}|\BE[g'(X)-Xg(X)]|. 
\end{displaymath}

Let the function $G$ satisfy the assumptions of Theorem~\ref{tmain}
and write $X:=\deltaKS(G)$. By the definition of the KS-integral we can
write $X\equiv X(\eta)$ as a measurable function of $\eta$.  Let
$g\in \mathbf{C}_{1,2}$. Then we have for $\lambda$-a.e.\ $x\in\BX$
and a.s.\ that
\begin{align}\label{e17.6}
  D_xg(X)=g(X(\eta+\Dirac_x))-g(X(\eta))=g(X+D_xX)-g(X).
\end{align}
Since $g$ is Lipschitz (by the boundedness of its first derivative)
and $X\in\dom D$ by Lemma~\ref{l3.3}, it follows that
$|D_xg(X)|\le |D_xX|$, so that $Dg(X)$ (considered as a function
on $\bN\times\BX$) is square integrable with respect to
$\BP_\eta\otimes\lambda$.  Since, moreover, it is clear that $g(X)$ is
square integrable, we have in particular that $g(X)\in\dom D$.
The integration by parts formula \eqref{e2.2} yields
that
\begin{equation}\label{eqn:partial_integration_Xg}
  \BE Xg(X)=\BE\int G_xD_xg(X)\,\lambda(dx).
\end{equation}
Since $G\in L^2(\BP_\eta\otimes\lambda)$
and $X\in\dom D$,
we obtain from the Lipschitz continuity of $g$ and the Cauchy--Schwarz inequality that
\begin{align}\label{e3.78}
  \BE\int |G_xD_xg(X)|\,\lambda(dx) \leq \BE\int |G_x| |D_xX|\,\lambda(dx) <\infty.
\end{align}
We have that
\begin{align*}
  D_xg(X) & = g(X+D_xX)-g(X) = \int_X^{X+D_xX} g'(t) \, dt
  = D_xX \int_0^1 g'(X+sD_xX) \, ds. 
\end{align*}
Our assumptions on $G$ allow to apply the commutation rule
\eqref{eq:delta1} to $D_xX$, yielding a.s.\ and for $\lambda$-a.e.\ $x$ that
\begin{align*}
  G_xD_xg(X)
  &=G_x D_xX \int_0^1 g'(X+sD_xX) \, ds\\
  & =\int_0^1 G_x (G_x + \deltaKS(D_xG))  g'(X+sD_xX) \, ds  \\
  & = \int_0^1 G_x^2  g'(X+sD_xX) \, ds + \int_0^1 G_x \deltaKS(D_xG)  g'(X+sD_xX) \, ds  \\
  & =:S_1(x)+S_2(x).
\end{align*}
In view of $|g'|\le 1$, \eqref{eq:delta1}, \eqref{eH2} and \eqref{e3.78}, we can note that 
\begin{align}\label{eabc}
& \BE \int\int^1_0 |G_x \deltaKS(D_xG)  g'(X+sD_xX)| \,ds\,\lambda(dx) 
 \leq \BE \int |G_x| (|D_xX| + |G_x|) \,\lambda(dx) <\infty.
\end{align}
We obtain
\begin{align*}
  |\BE[g'(X)-Xg(X)]|
  & \leq \bigg| \BE g'(X) \bigg( 1 - \int G_x^2 \, \lambda(dx) - \int D_xG_y D_y G_x \, 
    \lambda^2(d(x,y)) \bigg) \bigg| \\
  & \quad + \bigg| \BE  \int \big(g'(X) G_x^2-S_1(x)\big) \,\lambda(dx)  \bigg| \\
  & \quad + \bigg| \BE g'(X) \int D_xG_y D_y G_x \, \lambda^2(d(x,y)) 
    - \BE\int S_2(x)\,\lambda(dx) \bigg| \\
  & =: U_0 + U_1 + U_2.
\end{align*}
Since $\BE\deltaKS(G)^2=1$, Jensen's inequality and \eqref{eq:4} yield that 
\begin{align*}
  U_0 & \leq \BE\Big| 1- \int G_x^2 \, \lambda(dx)
        - \int D_xG_y D_y G_x \, \lambda^2(d(x,y)) \Big| \\
      & \leq \left(\BV \int G_x^2 \, \lambda(dx) \right)^{1/2}
        +  \left( \BV \int D_xG_y D_y G_x \, \lambda^2(d(x,y)) \right)^{1/2}.
\end{align*}
It follows from the Poincar\'e inequality (see
\cite[Section~18.3]{LastPenrose17}) that
\begin{displaymath}
  \BV \int G_x^2 \, \lambda(dx)
  \leq \BE \int \bigg(\int D_y(G_x^2)
  \, \lambda(dx)\bigg)^2 \, \lambda(dy) = T_1^2
\end{displaymath}
and
\begin{multline*}
  \BV \int D_xG_y D_y G_x \, \lambda^2(d(x,y))
  \leq \BE \int \bigg( \int D_z\big( D_xG_y D_y G_x \big)
     \, \lambda^2(d(x,y))\bigg)^2 \, \lambda(dz) = T_2^2,
\end{multline*}
whence
\begin{equation}\label{eq:BoundU0}
  U_0 \leq T_1+T_2.
\end{equation}

We now turn to $U_1$. We note first that, by $|g'|\le 1$ and \eqref{eH2},
\begin{align*}
\BE \int\int^1_0 G_x^2\big|g'(X)-g'(X+sD_xX)\big|\, ds \, \lambda(dx)<\infty.
\end{align*}
Because of
\begin{equation}\label{eqn:diff_g'}
g'(X+sD_xX)-g'(X) = s D_xX \int_0^1 g''(X+stD_xX) \, dt =: D_xX H(s,x)
\end{equation}
for $x\in\BX$ and $s\in[0,1]$, we have that
\begin{align}\notag
U_1& = \bigg|\BE\int\int_0^1 G_x^2(g'(X+sD_xX) - g'(X)) \, ds \, \lambda(dx)\bigg| \\ \notag
& = \bigg|\BE\int\int_0^1 G_x^2 D_xX H(s,x) \, ds \, \lambda(dx)\bigg|\\ 
& \le \bigg|\BE\int\int_0^1 G_x^2 G_x H(s,x) \, ds\, \lambda(dx)\bigg|
  \notag\\
& \quad   +\bigg|\BE \int\int_0^1 G_x^2 \deltaKS(D_xG)
    H(s,x) \, ds \, \lambda(dx)\bigg|,
    \label{e4.74}
\end{align}
where we have used the commutation rule \eqref{eq:delta1}
in the last step. To justify the linearity of the integration
we can assume without loss of generality that
\begin{align*}
T_3=\BE \int |G_x|^3\, \lambda(dx)<\infty
\end{align*}
and use that $|g''|\le 2$. The latter inequality yields that
$|H(s,x)|\leq 2s$ and
$$
\bigg| \BE \int\int_0^1 G_x^2 G_x H(s,x) \, ds \, \lambda(dx) \bigg| 
\leq \BE \int\int_0^1 |G_x|^3 |H(s,x)| \, ds \, \lambda(dx) 
\le T_3.
$$
To treat the term \eqref{e4.74} we first use
$|\deltaKS(D_xG)|\leq |D_xX|+|G_x|$ for $x\in\BX$ (see
\eqref{eq:delta1}), \eqref{eqn:diff_g'} and the preceding
integrability properties to conclude that
\begin{equation}\label{eqn:cond_lem}
\begin{split}
  & \BE\int\int^1_0 G_x^2 |\deltaKS(D_xG) H(s,x)|\, ds \, \lambda(dx) \\
  & \leq \BE\int\int^1_0 |G_x|^3 |H(s,x)|\, ds \, \lambda(dx) + \BE\int\int^1_0 G_x^2 |D_xX H(s,x)|\, ds \, \lambda(dx) \\
  & = \BE\int\int^1_0 |G_x|^3 |H(s,x)|\, ds \, \lambda(dx) +
  \BE\int\int^1_0 G_x^2 |g'(X)-g'(X+sD_xX)|\, ds \, \lambda(dx)
  <\infty.
\end{split}
\end{equation}
Therefore, we obtain
from Fubini's theorem that
\begin{align*}
  U_1\le T_3+ \int\int^1_0 |\BE G_x^2 \deltaKS(D_xG) H(s,x)|\, ds \, \lambda(dx).
\end{align*}
The expectation on the above right-hand side can be bounded
with Lemma~\ref{l3.2} applied to $H:=G_x^2H(s,x)$ and with $D_xG$ instead of $G$
(justified by \eqref{eH3}, \eqref{eH4}, and \eqref{eqn:cond_lem}).
This gives
\begin{align*}
  U_1&\le T_3+ 
  \int\int^1_0 \BE |D_xG_y| \big|D_y\big(G_x^2 H(s,x)\big)\big|\, ds\, \lambda^2(d(x,y)) \\
  & \leq T_3+\BE \int |D_xG_y| (|D_y(G_x^2)| + 2G_x^2)\,\lambda^2(d(x,y)),
\end{align*}
where we used \eqref{eq:1}, $|D_yH(s,x)+H(s,x)|\leq 2s$, and $|D_yH(s,x)|\leq 4s$.

Now we turn to the term $U_2$. Define $R_x:=\int^1_0g'(X+sD_xX)\,ds$, $x\in\BX$.
By the integrability property \eqref{eabc} and Fubini's theorem,
\begin{align*}
  \BE \int S_2(x)\,\lambda(dx) 
  =\int \BE \deltaKS(D_xG)G_xR_x \,\lambda(dx).
\end{align*}
By Lemma~\ref{l3.1}, whose assumptions are satisfied for
  $\lambda$-a.e.\ $x$ by \eqref{eH2}--\eqref{eH4} and $|g'|\le 1$,
and the product rule \eqref{eq:1},
\begin{align*}
  \BE& \int S_2(x)\,\lambda(dx) 
       =\int \int \BE D_xG_yD_y(G_xR_x) \,\lambda(dy) \, \lambda(dx)\\
     &=\int\int (\BE D_xG_yD_yG_xR_x+\BE D_xG_y(G_x+D_yG_x)D_yR_x) \,\lambda(dy) \, \lambda(dx),
\end{align*}
so that
\begin{align*}
  U_2 & \le  \int\bigg|\BE D_yG_x D_xG_y 
        \int_0^1 \big(g'(X+sD_xX) - g'(X)\big)\,ds \bigg|\, \lambda^2(d(x,y)) \\
      & \qquad + \int\bigg|\BE D_xG_y(G_x + D_yG_x) D_y\bigg(\int_0^1 g'(X+sD_xX) 
        \, ds\bigg)\bigg|\,\lambda^2(d(x,y)) \\
      & =: U_{2,1} + U_{2,2}.
\end{align*}
Here, the expectations exist for $\lambda^2$-a.e.\ $(x,y)$ because of $|g'|\leq 1$, \eqref{eH2} and \eqref{eH3}. In view of the definition of $T_4$ we can assume 
without loss of generality that
\begin{align}\label{e823}
\BE\int |D_yG_x D_xG_y G_x|\,\lambda^2(d(x,y))<\infty.
\end{align}
The commutation rule \eqref{eq:delta1} leads to
\begin{align*}
U_{2,1} & = \int\bigg|\BE D_yG_x D_xG_y D_xX 
\int_0^1\int_0^1 sg''(X+stD_xX) \, ds \, dt \bigg|\,\lambda^2(d(x,y)) \\
& \le  \int \bigg| \BE D_yG_x D_xG_y G_x 
\int_0^1\int_0^1 sg''(X+stD_xX)\,ds\,dt \bigg|\,\lambda^2(d(x,y)) \\
& \quad + \int \bigg|\BE D_yG_x D_xG_y\deltaKS(D_xG)\int_0^1\int_0^1 sg''(X+stD_xX) 
\,ds\,dt\bigg|\,\lambda^2(d(x,y)).
\end{align*}
The following computation as well as \eqref{eH3} and \eqref{eH4}
  allow us to apply Lemma \ref{l3.2} to the second term on the
  right-hand side. From the commutation rule \eqref{eq:delta1}, the
  boundedness of $g'$ and $g''$, \eqref{e823} and \eqref{eH3} we
  obtain
\begin{align*}
& \int \BE \bigg|D_yG_x D_xG_y\deltaKS(D_xG)\int_0^1\int_0^1 sg''(X+stD_xX) 
\,ds\,dt\bigg| \,\lambda^2(d(x,y)) \\
& \leq \int \BE |D_yG_x D_xG_y G_x| \,\lambda^2(d(x,y)) \\
& \quad + \int \BE \bigg|D_yG_x D_xG_y D_xX \int_0^1\int_0^1 sg''(X+stD_xX) 
\,ds\,dt\bigg| \,\lambda^2(d(x,y)) \\
& \leq \int \BE |D_yG_x D_xG_y G_x| \,\lambda^2(d(x,y)) \\
& \quad + \int \BE \bigg|D_yG_x D_xG_y \int_0^1 (g'(X+sD_xX) - g'(X))
\,ds\bigg| \,\lambda^2(d(x,y)) < \infty.
\end{align*}
Thus, we derive from Lemma \ref{l3.2} and $|g''|\le 2$ that
\begin{align*}
  U_{2,1}
  & \leq \BE \int |D_yG_x D_xG_y G_x| \, \lambda^2(d(x,y)) \\
  & \quad +  \int \int \BE \bigg|D_xG_z D_z
    \bigg(D_yG_x D_xG_y \int_0^1\int_0^1 sg''(X+stD_xX) \, ds \, dt\bigg) \bigg| 
    \, \lambda(dz)\,\lambda^2(d(x,y))  \allowdisplaybreaks  \\
  & \leq \BE \int |D_yG_x D_xG_y G_x| \, \lambda^2(d(x,y)) \\
  & \quad +  \BE \int |D_xG_z| \big(|D_z\big(D_yG_x D_xG_y\big)| 
    + 2|D_yG_x D_xG_y| \big) \, \lambda^3(d(x,y,z)),
\end{align*}
where we used \eqref{eq:1} in the last step. Similarly as in \eqref{e17.6}, we derive 
\begin{align}
  \notag
  D_y\bigg(\int_0^1
  &g'(X+sD_xX) \, ds\bigg) \\
  \notag
  & = \int_0^1 (g'(X+sD_xX+D_yX+sD^2_{x,y}X) - g'(X+sD_xX)) \, ds \\
  \notag
  & = \int_0^1 \int_0^1 (D_yX+sD^2_{x,y}X) g''(X+sD_xX+ t (D_yX+sD^2_{x,y}X)) \, dt \, ds\\
  \label{e9876}
  & = \int_0^1 (D_yX+sD^2_{x,y}X) R(s,x,y) \, ds
\end{align}
for $x,y\in\BX$ with
$$
R(s,x,y):=\int_0^1 g''(X+sD_xX+ t (D_yX+sD^2_{x,y}X)) \, dt.
$$
By assumptions \eqref{eH2}-\eqref{eH5} we can use
the commutation rule \eqref{eq:delta1} twice to obtain 
that
\begin{align*}
  D^2_{x,y}X=D_y(D_x\deltaKS(G))=D_y(G_x+\deltaKS(D_xG))=D_yG_x+D_xG_y+\deltaKS(D^2_{x,y}G)
\end{align*}
a.s.\ and for $\lambda^2$-a.e.\ $(x,y)$, while $D_yX=G_y+\deltaKS(D_yG)$
a.s.\ and for $\lambda$-a.e.\ $y$.  Therefore, \eqref{e9876} equals
\begin{align*}
  \int_0^1 (G_y+\deltaKS(D_yG)+s (D_xG_y+D_yG_x + \deltaKS(D^2_{x,y}G)))R(s,x,y)\,ds.
\end{align*}
For $s\in[0,1]$ one has
\begin{equation}\label{eqn:e9877}
  \begin{split}
    \big|(G_y&+\deltaKS(D_yG)+s (D_xG_y+D_yG_x + \deltaKS(D^2_{x,y}G)))R(s,x,y)\big| \\
    & = \big| (D_yX+s D^2_{x,y}X) R(s,x,y) \big| \\
    & = \bigg| \int_0^1 (D_yX+s D^2_{x,y}X) g''(X+sD_xX+ t (D_yX+sD^2_{x,y}X)) \, dt \bigg| \\
    & = \big| g'(X+sD_xX+ D_yX+sD^2_{x,y}X) - g'(X+sD_xX) \big| \leq 2,
  \end{split}
\end{equation}
whence
$$
\bigg| \int_0^1 (\deltaKS(D_yG)+s \deltaKS(D^2_{x,y}G)) R(s,x,y)\,ds \bigg| \leq 2 + \bigg| \int_0^1 (G_y+s (D_xG_y+D_yG_x))R(s,x,y)\,ds \bigg|.
$$
Since $|R(s,x,y)|\leq 2$, 
$$
\bigg| \int_0^1 (G_y+s (D_xG_y+D_yG_x))R(s,x,y)\,ds \bigg| \leq 2 |G_y| + |D_xG_y+D_yG_x|.
$$
Because of the assumption $T_4<\infty$, this yields
\begin{multline*}
  \int \bigg| \BE D_xG_y (G_x+D_yG_x) \int_0^1 (G_y+s (D_xG_y+D_yG_x))R(s,x,y)\,ds  \bigg| \, \lambda^2(d(x,y)) \\
  \leq \int \BE  |D_xG_y (G_x+D_yG_x)|  (2 |G_y| + |D_xG_y+D_yG_x|) \, \lambda^2(d(x,y))<\infty.
\end{multline*}
Together with \eqref{eH2} and \eqref{eH3}, we deduce from \eqref{eqn:e9877} that
\begin{multline}\label{eqn:cond_lem3.2}
  \BE \int_0^1 \big| D_xG_y (G_x+D_yG_x) (\deltaKS(D_yG)+s \deltaKS(D^2_{x,y}G)) R(s,x,y) \big| \,ds \\
  \leq \BE |D_xG_y (G_x+D_yG_x)|  (2 + 2 |G_y| + |D_xG_y+D_yG_x|) <\infty
\end{multline}
for $\lambda^2$-a.e.\ $(x,y)$. Hence, we have shown that
\begin{align*}\notag
U_{2,2} & \leq \BE \int \big|(G_x + D_yG_x) D_xG_y\big| (2|G_y|+|D_xG_y+D_yG_x|) \, \lambda^2(d(x,y)) \\
& \quad +   \int_0^1 \int \big|\BE (G_x + D_yG_x) D_xG_y\deltaKS(D_yG+sD^2_{x,y}G)   
 R(s,x,y) \big| \, \lambda^2(d(x,y))\, ds.
\end{align*}
By Lemma~\ref{l3.2}, which can be applied due to
\eqref{eqn:cond_lem3.2}, the second term on the right-hand side can be
further bounded by
\begin{align*}
&  \int_0^1 \int \big|\BE(D_yG_z+sD^2_{x,y}G_z) D_z((G_x + D_yG_x) D_xG_y R(s,x,y))
 \big|\,\lambda^3(d(x,y,z))\, ds\\
& \quad \leq 2\,\BE \int (|D_yG_z|+|D^2_{x,y}G_z|) \big(|D_z\big((G_x + D_yG_x) D_xG_y\big)| + 
2 |(G_x + D_yG_x) D_xG_y| \big) \\
& \qquad \qquad \qquad \times \lambda^3(d(x,y,z)).
\end{align*}

Combining the previous bounds, we see that
\begin{align*}
U_1+U_2 \leq & \BE \int |G_x|^3\, \lambda(dx) + \BE \int \big( 2 |D_xG_y D_yG_x G_x| 
+ |D_xG_y (D_yG_x)^2| + 2G_x^2 |D_xG_y| \\
& \qquad \qquad \qquad \qquad \qquad + \big| (G_x + D_yG_x)  D_xG_y\big| (2|G_y|+|D_xG_y+D_yG_x|) \\
& \qquad \qquad \qquad \qquad \qquad + |D_yG_x D_xG_y G_x| \big)\, \lambda^2(d(x,y)) \\
& + \BE \int 2(|D_yG_z|+|D^2_{x,y}G_z|) \big(|D_z\big((G_x + D_yG_x) D_xG_y\big)| + 2|(G_x + D_yG_x) D_xG_y| \big) \\
& \qquad \quad +   |D_xG_z| \big(|D_z\big(D_yG_x D_xG_y\big)| + 2|D_yG_x D_xG_y| \big)\, \lambda^3(d(x,y,z)) \\
= & T_3 + T_4 + T_5,
\end{align*}
which together with \eqref{eq:BoundU0} completes the proof.

\section{Proof for the Kolmogorov distance in
  Theorem~\ref{tmain}} \label{sec:Kolmogorov}

We prepare the proof of the second part of Theorem \ref{tmain} by two lemmas. Since we consider iterated KS-integrals in the following, we
indicate the integration variable as a subscript, i.e., write
$\deltaKS_x$ to denote the KS-integral with respect to $x$.

\begin{lemma}\label{lem:IteratedSkorohodIntegral}
  Let $h:\mathbf{N}\times\BX^2\to\R$ be measurable and such that
  \begin{multline}\label{eqn:integrability_h}
    \BE \int h(x,y)^2 \, \lambda^2(d(x,y))
      + \BE \int (D_zh(x,y))^2 \, \lambda^3(d(x,y,z)) \\
    + \BE\int (D^2_{z,w}h(x,y))^2 \, \lambda^4(d(x,y,z,w)) <\infty.
  \end{multline}
  \begin{itemize}
  \item [{\rm (i)}] Then, $\deltaKS_x(\deltaKS_y(h(x,y)))$ is well defined and
    \begin{multline*}
      \BE\big[ \deltaKS_x(\deltaKS_y(h(x,y)))^2 \big] 
      \leq 3 \BE \int h(x,y)^2 \, \lambda^2(d(x,y))
      + 3 \BE \int \big(D_zh(x,y)\big)^2 \, \lambda^3(d(x,y,z)) \\
      + 2\BE\int \big(D^2_{w,z}h(x,y)\big)^2 \, \lambda^4(d(x,y,z,w)).
    \end{multline*}
  \item [{\rm (ii)}] If $H\in L^2(\BP_\eta)$ is such that
    $D_xH\in L^2(\BP_\eta)$ for $\lambda$-a.e.\ $x$,
    $D^2_{x,y}H\in L^2(\BP_\eta)$ for $\lambda^2$-a.e.\ $(x,y)$ and
    \begin{equation}\label{eqn:Integrability_IteratedSkorohodIntegral}
      \BE \int |D^2_{x,y}H h(x,y)| \, \lambda^2(d(x,y))<\infty,
    \end{equation}
    then
    $$
    \BE \int D^2_{x,y}H h(x,y) \, \lambda^2(d(x,y)) = \BE\big[ \deltaKS_x(\deltaKS_y(h(x,y))) H \big].
    $$
  \end{itemize}
\end{lemma}
\begin{proof}
  First, let us assume that all KS-integrals are
  well defined. By applying iteratively
  \cite[Corollary~2.4]{las:pec:sch16} and \eqref{eq:delta1}, we have
  \begin{align*}
    & \BE\big[ \deltaKS_x(\deltaKS_y(h(x,y)))^2 \big] \\
    & \leq \BE \int \deltaKS_y(h(x,y))^2 \, \lambda(dx)
      + \BE\int (D_z\deltaKS_y(h(x,y)))^2 \, \lambda^2(d(x,z)) \\
    & \leq \BE \int \deltaKS_y(h(x,y))^2 \, \lambda(dx)
      + 2\BE\int h(x,z)^2 \, \lambda^2(d(x,z))
      + 2\BE\int \deltaKS_y(D_zh(x,y))^2 \, \lambda^2(d(x,z)) \\
    & \leq \BE \int h(x,y)^2 \, \lambda^2(d(x,y))
      + \BE \int \big(D_zh(x,y)\big)^2 \, \lambda^3(d(x,y,z))
      + 2\BE\int h(x,z)^2 \, \lambda^2(d(x,z)) \\
    & \quad + 2\BE\int \big(D_zh(x,y)\big)^2 \, \lambda^3(d(x,y,z))
      + 2\BE\int \big(D^2_{w,z}h(x,y)\big)^2 \, \lambda^4(d(x,y,z,w)) \\
    & = 3 \BE \int h(x,y)^2 \, \lambda^2(d(x,y))
      + 3 \BE \int \big(D_zh(x,y)\big)^2 \, \lambda^3(d(x,y,z)) \\
    & \quad + 2\BE\int \big(D^2_{w,z}h(x,y)\big)^2 \, \lambda^4(d(x,y,z,w)).
  \end{align*}
  Since, by \eqref{eqn:integrability_h}, the right-hand side is
  finite, all involved KS-integrals are well defined by
  \cite[Proposition~2.3]{las:pec:sch16}.
	
  Because of \eqref{eqn:Integrability_IteratedSkorohodIntegral} and
  Fubini's theorem, we have
  $$
  J:=\BE \int D^2_{x,y}H h(x,y) \, \lambda^2(d(x,y)) = \int \int \BE D^2_{x,y}H h(x,y) \, \lambda(dy) \, \lambda(dx).
  $$
  For $\lambda$-a.e.\ $x$ our assumptions imply
  $D_xH\in L^2(\BP_\eta)$, $D^2_{x,y}H\in L^2(\BP_\eta)$ for
  $\lambda$-a.e.\ $y$ as well as
  $$
  \BE \int h(x,y)^2 \, \lambda(dy)<\infty \quad \text{and} \quad \BE \int (D_zh(x,y))^2 \, \lambda^2(d(y,z))<\infty.
  $$
  Thus, it follows from Lemma \ref{l3.1} that
  $$
  J = \int  \BE D_xH \deltaKS_y(h(x,y)) \, \lambda(dx).
  $$
  Since $H\in L^2(\BP_\eta)$, $D_xH\in L^2(\BP_\eta)$ for
  $\lambda$-a.e.\ $x$ and combining \eqref{eqn:integrability_h} and
  \cite[Corollary~2.4]{las:pec:sch16} as in the proof of part (i)
  yields
  $$
  \BE \int \deltaKS_y(h(x,y))^2 \, \lambda(dx)<\infty 
  \quad
  \text{and}
  \quad 	\BE \int (D_z\deltaKS_y(h(x,y)))^2 \, \lambda^2(d(x,z))<\infty,
  $$
  a further application of Lemma \ref{l3.1} leads to
  $$
  J = \BE H \deltaKS_x(\deltaKS_y(h(x,y))),
  $$
  which concludes the proof of part (ii).
\end{proof}

For $a\in\R$, let $f_a$ be a solution of the Stein equation
\begin{equation}\label{eqn:SteinEquationKolmogorov}
  f_a'(u) - u f_a(u) = \mathbf{1}\{u\leq a\} - \Phi(a), \quad u\in\R,
\end{equation}
where $\Phi$ is the distribution function of the standard normal
distribution. Note that $f_a$ is continuously differentiable on
$\R\setminus\{a\}$. Thus, we use the convention that $f_a'(a)$
is the left-sided limit of $f_a'$ in $a$. For the following lemma we
refer the reader to \cite[Lemma~2.2 and Lemma~2.3]{ChenGoldsteinShao}.

\begin{lemma}\label{lem:SteinKolmogorov}
  For each $a\in\R$ there exists a unique bounded solution $f_a$ of
  \eqref{eqn:SteinEquationKolmogorov}. This function satisfies:
  \begin{itemize}
  \item [{\rm (i)}] $u\mapsto u f_a(u)$ is non-decreasing;
  \item [{\rm (ii)}] $|uf_a(u)|\leq 1$ for all $u\in\R$;
  \item [{\rm (iii)}] $|f_a'(u)|\leq 1$ for all $u\in\R$.
  \end{itemize}
\end{lemma}

Now we are ready for the proof for the Kolmogorov distance. It
  combines the approach for the Wasserstein distance with arguments
  from \cite{LPY20}, which refined ideas previously used in
  \cite{eich:tha16} and \cite{schulte16}. Indeed, for the normal
  approximation of Poisson functionals in Kolmogorov distance the
  Malliavin-Stein method was first used in \cite{schulte16}. One of
  the terms in the bound was removed in \cite{eich:tha16} and two more
  in \cite{LPY20}. The innovation of \cite{LPY20}, which was inspired by the proof of Theorem 2.2 in \cite{ShaoZhang2019} and which we also
  employ in the following, is to exploit the monotonicity of
  $u\mapsto u f_a(u)$ and $u\mapsto \mathbf{1}\{u\leq a\}$.

\begin{proof}[Proof for the Kolmogorov distance in Theorem~\ref{tmain}]
  Throughout the proof we can assume without loss of generality that
  $T_1,T_2,T_6,T_7,T_8,T_9<\infty$. Let $a\in\R$, and let $f_a$ be the
  solution of \eqref{eqn:SteinEquationKolmogorov} from
  Lemma~\ref{lem:SteinKolmogorov}.  For $X:=\deltaKS(G)$ we have
  $f_a(X)\in\operatorname{dom} D$ (since $|f'_a|\leq 1$ and
  $X\in \operatorname{dom} D$), whence the integration by parts rule
  \eqref{e2.2} yields similarly as in
  \eqref{eqn:partial_integration_Xg} that
\begin{displaymath}
  \BE\big[f_a'(X) - X f_a(X)\big]
  =  \BE\Big[ f_a'(X) - \int G_x D_xf_a(X) \, \lambda(dx)\Big].
\end{displaymath}
Together with
\begin{displaymath}
  D_xf_a(X) = f_a(X+D_xX) - f_a(X) = \int_0^{D_xX} f_a'(X+s) \, ds,
\end{displaymath}
we obtain
\begin{align*}
  \BE\big[f_a'(X) - X f_a(X)\big]
  & = \BE f_a'(X) \Big( 1- \int G_x D_xX \, \lambda(dx) \Big) \\
  & \qquad - \BE \int \int_0^{D_xX} \big(f_a'(X+s) - f_a'(X)\big) \, ds  \ G_x \, \lambda(dx) \\
  & =: I_1-I_2,
\end{align*}
where the decomposition into $I_1$ and $I_2$ is allowed due to $|f_a'|\leq 1$ and \eqref{e3.78}. The commutation rule \eqref{eq:delta1} yields
$$
I_1 = \BE f_a'(X) \Big( 1 - \int G_x^2 \, \lambda(dx)        - \int G_x \deltaKS(D_xG) \, \lambda(dx) \Big).
$$
From Fubini's theorem, which is applicable because of $|f_a'|\leq 1$
and \eqref{e3.78}, and Lemma~\ref{l3.1} it follows that
\begin{align*}
\BE f_a'(X) \int G_x \deltaKS(D_xG) \, \lambda(dx) & = \int \BE f_a'(X) G_x \deltaKS(D_xG)  \, \lambda(dx) \\
& = \int \int \BE D_xG_y D_y\big(f_a'(X) G_x\big) \, \lambda(dy) \, \lambda(dx).
\end{align*}
The use of Lemma~\ref{l3.1} is justified by
$f_a'(X)G_x\in L^2(\BP_\eta)$ for $\lambda$-a.e.\ $x$ and
$D_y(f_a'(X)G_x)\in L^2(\BP_\eta)$ for $\lambda^2$-a.e.\ $(x,y)$,
which are consequences of $|f_a'|\leq 1$, \eqref{eH2} and \eqref{eH3},
as well as \eqref{eH3} and \eqref{eH4}. From \eqref{eq:1} we derive
$$
D_y\big(f_a'(X) G_x\big) = f_a'(X) D_yG_x + D_yf_a'(X) (G_x+D_yG_x). 
$$
Combining this with $|f_a'|\leq 1$, \eqref{eH3} and \eqref{eH_mixed}, we see that
\begin{equation}\label{eqn:check_Fubini}
\begin{split}
& \int \int \BE \big|D_xG_y D_y\big(f_a'(X) G_x\big)\big| \, \lambda(dy) \, \lambda(dx) \\
& \leq \int \BE \big|f_a'(X) D_xG_y D_yG_x \big| \, \lambda^2(d(x,y)) + \int \BE \big| D_yf_a'(X) D_xG_y  (G_x+D_yG_x)  \big| \, \lambda^2(d(x,y)) \\
& \leq \BE \int |D_xG_y D_yG_x|  \, \lambda^2(d(x,y))  + 2 \BE \int (|D_xG_y G_x| + |D_xG_y D_yG_x|) \, \lambda^2(d(x,y)) < \infty .
\end{split}
\end{equation}
By Fubini's theorem, this makes it possible to rewrite $I_1$ as
\begin{align*}
  I_1 & = \BE f_a'(X) \big( 1 - \int G_x^2 \, \lambda(dx)
        - \int D_yG_x D_xG_y \, \lambda^2(d(x,y)) \big) \\
      & \quad - \BE \int D_yf_a'(X) (G_x+D_yG_x) D_xG_y \, \lambda^2(d(x,y)) 
= : I_{1,1} - I_{1,2}.
\end{align*}
It follows, as in the proof for the Wasserstein distance, that
$$
|I_{1,1}| \leq T_1+T_2.
$$
As shown in \eqref{eqn:check_Fubini}, we can apply Fubini's theorem to
$I_{1,2}$, so that
$$
I_{1,2} = \int \BE D_yf_a'(X) \int (G_x+D_yG_x) D_xG_y \, \lambda(dx) \, \lambda(dy).
$$
The boundedness of $f_a'$ implies that $|f_a'(X)| \leq 1$ and
$|D_yf_a'(X)|\leq 2$ for $\lambda$-a.e.\ $y$, while
$y\mapsto \int (G_x+D_yG_x) D_xG_y \, \lambda(dx)$ satisfies
\eqref{eH2} and \eqref{eH3} because of $T_6<\infty$. Thus,
Lemma~\ref{l3.1} shows that
$$
I_{1,2}  = \BE f_a'(X) \deltaKS_y\Big(
\int (G_x+D_yG_x) D_xG_y \, \lambda(dx)\Big).
$$
Together with $|f_a'|\leq 1$ and Jensen's inequality, we obtain that
\begin{align*}
  |I_{1,2}|
  & \leq \BE \big|\deltaKS_y\Big(
  \int (G_x+D_yG_x) D_xG_y \, \lambda(dx)\Big)\big| \\
  & \leq \bigg(\BE \deltaKS_y\Big( \int (G_x+D_yG_x) D_xG_y
  \, \lambda(dx)\Big)^2 \bigg)^{1/2}.
\end{align*}
It follows from \cite[Corollary~2.4]{las:pec:sch16} that
\begin{align*}
  \BE
  &\deltaKS_y\Big( \int (G_x+D_yG_x) D_xG_y \,
    \lambda(dx)\Big)^2 \\
  & \leq \BE \int \bigg( \int (G_x+D_yG_x) D_xG_y
    \, \lambda(dx)\bigg)^2 \, \lambda(dy) \\
  & \qquad + \BE \int \bigg( \int D_z \big((G_x+D_yG_x) D_xG_y\big) \,
    \lambda(dx)\bigg)^2 \, \lambda^2(d(y,z)) = T_6^2.
\end{align*}

In the sequel, we focus on $I_2$. By
\eqref{eqn:SteinEquationKolmogorov}, the inner integral in $I_2$
equals
\begin{displaymath}
  \int_0^{D_xX} \Big((X+s) f_a(X+s) - X f_a(X)
  + \mathbf{1}\{ X+s \leq a\} - \mathbf{1}\{ X \leq a\} \Big) \, ds.
\end{displaymath}
Since $u \mapsto u f_a(u)$ is non-decreasing (see
Lemma~\ref{lem:SteinKolmogorov} (i)) and
$u\mapsto \mathbf{1}\{u\leq a\}$ is non-increasing, we derive by
considering the cases $D_xX\geq 0$ and $D_xX<0$ separately that
\begin{align*}
  \bigg| \int_0^{D_xX} \Big((X+s) f_a(X+s) - X f_a(X)\Big) \, ds
  \bigg| 
  & \leq D_xX \Big((X+D_xX) f_a(X+D_xX) - X f_a(X)\Big) \\
  & = D_xX D_x(Xf_a(X))
\end{align*}
and
\begin{align*}
  \bigg| \int_0^{D_xX} \Big(\mathbf{1}\{ X+s \leq a\}
  - \mathbf{1}\{ X \leq a\} \Big) \, ds \bigg|
  & \leq -D_xX \Big(\mathbf{1}\{ X+D_xX \leq a\}
    - \mathbf{1}\{ X \leq a\} \Big)\\
  & = -D_xX D_x\mathbf{1}\{ X \leq a\}.
\end{align*}
Combining these estimates with \eqref{eq:delta1} leads to
\begin{align*}
  |I_2| & \leq \BE \int D_xX D_x\big(Xf_a(X)
          - \mathbf{1}\{X\leq a\}\big)  |G_x| \, \lambda(dx) \\ 
        & = \BE \int D_x(Xf_a(X) - \mathbf{1}\{X\leq a\}) G_x |G_x|
          \, \lambda(dx) \\
        & \quad + \BE \int \deltaKS(D_xG) D_x
          \big(Xf_a(X) - \mathbf{1}\{X\leq a\}\big)
          |G_x| \, \lambda(dx) 
          =: I_{2,1} + I_{2,2}.
\end{align*}
The decomposition into two integrals on the right-hand side is allowed as can be seen from the following argument. From Lemma~\ref{lem:SteinKolmogorov} (ii) we know that 
\begin{equation}\label{eqn:bound_solution_indicator}
|uf_a(u)-\mathbf{1}\{u\leq a\}|\leq 2 \quad \text{for all} \quad u\in\R.
\end{equation}
Together with \eqref{eH2}, we see that
\begin{align*}
\BE \int \big|D_x(Xf_a(X) - \mathbf{1}\{X\leq a\}) G_x |G_x|\big| \, \lambda(dx) \leq 4 \BE \int G_x^2 \, \lambda(dx) <\infty.
\end{align*}
It follows from \eqref{eqn:bound_solution_indicator}, the Cauchy--Schwarz inequality, \cite[Corollary~2.4]{las:pec:sch16} and \eqref{eH2}--\eqref{eH4} that 
\begin{align*}
& \BE \int \big|\deltaKS(D_xG) D_x \big(Xf_a(X) - \mathbf{1}\{X\leq a\}\big) |G_x|\big| \, \lambda(dx) \leq 4 \BE \int |\deltaKS(D_xG) G_x| \, \lambda(dx) \\
& \leq 4 \bigg(\BE \int \deltaKS(D_xG)^2 \, \lambda(dx) \bigg)^{1/2} \bigg(\BE \int G_x^2 \, \lambda(dx) \bigg)^{1/2} \\
& \leq 4 \bigg(\BE \int (D_xG_y)^2 \, \lambda^2(d(x,y)) + \BE \int (D^2_{x,z}G_y)^2 \, \lambda^3(d(x,y,z)) \bigg)^{1/2} \bigg( \BE \int G_x^2 \, \lambda(dx) \bigg)^{1/2}<\infty.
\end{align*}
Thus, the integrals $I_{2,1}$ and $I_{2,2}$ are well defined and finite. Moreover, we can interchange expectation and integration in $I_{2,1}$ and $I_{2,2}$ by Fubini's theorem.

We deduce from \eqref{eqn:bound_solution_indicator} for
$Z:=Xf_a(X) - \mathbf{1}\{X\leq a\}$ that
\begin{equation}\label{eqn:condition_partial}
  |Z| \leq 2, \quad |D_xZ|\leq 4 \quad \text{for} \quad \lambda\text{-a.e.}\ x \quad
  \text{and} \quad |D^2_{x,y}Z|\leq 8 \quad \text{for} \quad \lambda^2\text{-a.e.} \ (x,y).
\end{equation}
Note that $\BE \int G_x^4 \, \lambda(dx) <\infty$ since
$T_7<\infty$. Together with \eqref{eH_mixed_2}, we see that
$\BX\ni x\mapsto G_x|G_x|$ satisfies the integrability
conditions \eqref{eH2} and \eqref{eH3} and that
$G|G|\in\operatorname{dom}\deltaKS$. Thus, Lemma~\ref{l3.1} with $G$
replaced by $G |G|$ implies
\begin{displaymath}
  I_{2,1} = \BE \big(Xf_a(X) - \mathbf{1}\{X\leq a\}\big)
  \deltaKS(G |G|).
\end{displaymath}

Since $ D_x\big(Xf_a(X) - \mathbf{1}\{X\leq a\}\big)|G_x|\in L^2(\BP_\eta)$ 
for $\lambda$-a.e.\ $x$ and $D_y( D_x\big(Xf_a(X) - \I\{X\leq a\}\big)|G_x|)\in L^2(\BP_\eta)$ 
for $\lambda^2$-a.e.\ $(x,y)$, Lemma \ref{l3.1} and the product rule \eqref{eq:1} yield
\begin{align*}
  I_{2,2}
  & = \BE \int D_xG_y D^2_{x,y}\big(Xf_a(X)
    - \mathbf{1}\{X\leq a\}\big)  (D_y|G_x|+|G_x|) \,
    \lambda^2(d(x,y)) \\
  & \qquad +  \BE \int D_xG_y D_x\big(Xf_a(X)
    - \mathbf{1}\{X\leq a\}\big)  D_y|G_x| \,
    \lambda^2(d(x,y)).
\end{align*}
The decomposition of $I_{2,2}$ into two integrals is justified since
it follows from \eqref{eqn:bound_solution_indicator}, \eqref{eH3} and
\eqref{eH_mixed} that
\begin{equation}\label{eqn:absolute_integrable_partial}
\begin{split}
& \BE \int |D_xG_y D^2_{x,y}\big(Xf_a(X) - \mathbf{1}\{X\leq a\}\big)  (D_y|G_x|+|G_x|)| \, \lambda^2(d(x,y)) \\
& \leq 8 \BE \int |D_xG_y G_x| + (D_xG_y)^2 \, \lambda^2(d(x,y)) < \infty
\end{split}
\end{equation}
and
\begin{equation}\label{eqn:absolute_integrable_partial_II}
\begin{split}
& \BE \int |D_xG_y D_x\big(Xf_a(X) - \mathbf{1}\{X\leq a\}\big)  D_y|G_x|| \, \lambda^2(d(x,y)) \\
& \leq 4 \BE \int(D_xG_y)^2 \, \lambda^2(d(x,y)) < \infty.
\end{split}
\end{equation}
Note that $h(x,y):=D_xG_y (D_y|G_x|+|G_x|)$ satisfies
\eqref{eqn:integrability_h} because of $T_9<\infty$, so that
$\deltaKS_x(\deltaKS_y(h(x,y)))$ is well defined by
Lemma~\ref{lem:IteratedSkorohodIntegral} (i). Together with
\eqref{eqn:condition_partial} and
\eqref{eqn:absolute_integrable_partial} it follows from Lemma
\ref{lem:IteratedSkorohodIntegral} (ii) that
\begin{multline*}
  \BE \int D_xG_y D^2_{x,y}\big(Xf_a(X) - \mathbf{1}\{X\leq a\}\big)  (D_y|G_x|+|G_x|) \, \lambda^2(d(x,y)) \\
  = \BE \big(Xf_a(X) - \mathbf{1}\{X\leq a\}\big) \deltaKS_x\big(\deltaKS_y(D_xG_y (D_y|G_x|+|G_x|))\big).
\end{multline*}
Because of $T_8<\infty$ we see that
$$
\BE \int \Big( \int D_xG_y D_y|G_x| \, \lambda(dy) \Big)^2 \, \lambda(dx)<\infty
$$
and recall \eqref{eH_mixed_3}, whence
$\BX\ni x\mapsto \int D_xG_y D_y|G_x| \, \lambda(dy)$ satisfies
the integrability assumptions \eqref{eH2} and \eqref{eH3} and belongs
to $\operatorname{dom}\deltaKS$. By \eqref{eqn:condition_partial},
\eqref{eqn:absolute_integrable_partial_II}, Fubini's theorem and
Lemma~\ref{l3.1},
\begin{align*}
\BE \int &D_xG_y D_x\big(Xf_a(X) - \mathbf{1}\{X\leq a\}\big)  D_y|G_x| \, \lambda^2(d(x,y)) \\
& = \int \BE D_x\big(Xf_a(X) - \mathbf{1}\{X\leq a\}\big) \int D_xG_y D_y|G_x| \, \lambda(dy) \, \lambda(dx) \\
& = \BE \big(Xf_a(X) - \mathbf{1}\{X\leq a\}\big) \deltaKS_x\Big( \int D_xG_y D_y|G_x| \, \lambda(dy) \Big).
\end{align*}
We have shown that
\begin{align*}						
I_{2,2} & = \BE \big(Xf_a(X) - \mathbf{1}\{X\leq a\}\big)
            \deltaKS_x\big(\deltaKS_y(D_xG_y (D_y|G_x|+|G_x|))\big) \\
          & \qquad + \BE \big(Xf_a(X) - \mathbf{1}\{X\leq a\}\big)
            \deltaKS_x\Big( \int D_xG_y D_y|G_x| \, \lambda(dy) \Big).
\end{align*}

Now \eqref{eqn:bound_solution_indicator} and Jensen's
inequality yield that
\begin{displaymath}
|I_{2,1}| \leq 2 \BE |\deltaKS(G |G|)|
  \leq 2 \sqrt{ \BE \deltaKS(G |G|)^2 } 
\end{displaymath}
and that
\begin{align*}
  |I_{2,2}|
   &\leq 2 \bigg(\BE \deltaKS_x\big(\deltaKS_y(D_xG_y
    (D_y|G_x|+|G_x|))\big)^2 \bigg)^{1/2} \\
 &\qquad + 2\bigg(\BE \deltaKS_x\Big( \int D_xG_y D_y|G_x| \,
                  \lambda(dy) \Big)^2 \bigg)^{1/2}.
\end{align*}
By \eqref{eq:4}, we have
\begin{displaymath}
\BE \deltaKS(G |G|)^2
  = \BE \int G_x^4 \, \lambda(dx)
  + \BE \int D_x(G_y|G_y|) D_y(G_x|G_x|) \, \lambda^2(d(x,y)) = T_7^2
\end{displaymath}
and
\begin{align*}
  \BE &\deltaKS_x\Big(
  \int D_xG_y D_y|G_x| \,\lambda(dy) \Big)^2 \\
  & \leq \BE \int \bigg( \int D_xG_y D_y|G_x| \,
    \lambda(dy) \bigg)^2 \, \lambda(dx) \\
  & \quad + \BE \int D_x\bigg( \int D_zG_y D_y|G_z| \,
    \lambda(dy) \bigg)  D_z\bigg( \int D_xG_y D_y|G_x| \,
    \lambda(dy) \, \bigg) \, \lambda^2(d(x,z)) \\
  & = T_8^2.
\end{align*}
From Lemma~\ref{lem:IteratedSkorohodIntegral} (i), whose assumptions
are satisfied due to $T_9<\infty$, it follows that
\begin{align*}
\BE \deltaKS_x\big(\deltaKS_y(D_xG_y &(D_y|G_x|+|G_x|))\big)^2 \\
& \leq 3\BE \int (D_xG_y)^2 (D_y|G_x|+|G_x|)^2\, \lambda^2(d(x,y)) \\
& \qquad + 3 \BE \int \big( D_z\big( D_xG_y (D_y|G_x|+|G_x|) \big) \big)^2 \, \lambda^3(d(x,y,z)) \\
& \qquad +2 \BE \int \big( D^2_{z,w}\big( D_xG_y (D_y|G_x|+|G_x|) \big) \big)^2 \,\lambda^4(d(x,y,z,w))
= T_9^2,
\end{align*}
which completes the proof.
\end{proof}

\section{Poisson embedding}\label{sec:Poisson_embedding}

In this section we consider a Poisson process $\eta$ on
$\BX:=\R^d\times\R_+$, whose intensity measure 
$\lambda$ is the product of the Lebesgue measure $\lambda_d$ on
$\R^d$ and the Lebesgue measure $\lambda_+$ on $\R_+$.  We fix a
measurable mapping $\varphi\colon\R^d\times\bN\to[0,\infty]$, 
where the value $\infty$ is allowed for technical convenience.
Then
\begin{align}\label{eqn:xi}
  \xi:=\int \I\{s\in\cdot\}
  \I\{x\le \varphi(s,\eta-\delta_{(s,x)})\}\,\eta(d(s,x)) 
\end{align}
is a point process on $\R^d$. (At this stage it might not be locally
finite.)  Let $u\colon\R^d\to\R$ be a measurable function, and define
$G\colon \bN\times\BX\to\R$ by
\begin{align*}
  G_{(s,x)}(\mu):=u(s)\I\{x\le \varphi(s,\mu)\},
  \quad (\mu,(s,x))\in\bN\times\BX.
\end{align*}
Under suitable integrability assumptions we then have
\begin{align*}
  \deltaKS(G)=\int u(s)\I\{x\le \varphi(s,\eta-\delta_{(s,x)})\}\,\eta(d(s,x))
  -\int u(s)\I\{x\le \varphi(s,\eta)\}\,\lambda(d(s,x)),
\end{align*}
that is, 
\begin{align*}
  \deltaKS(G)=\int u(s)\,\xi(ds)-\int u(s)\varphi(s,\eta)\,ds.
\end{align*}
This can be interpreted as integral of $u$ with respect to the
\emph{compensated} point process $\xi$.
To make the dependence on $u$ more visible, we abuse our
notation and write $\deltaKS(u):=\deltaKS(G)$, whenever this integral is defined
pathwise.

Under certain assumptions, it can be expected that
the standardised $\deltaKS(u)$ is getting close to a 
normal distribution. To establish an asymptotic scenario, 
we take a Borel set $B\subset\R^d$ with $\lambda_d(B)<\infty$ and
define the function $u_B\colon\R^d\to\R$ by $u_B(s):=\I\{s\in B\}u(s)$.
Then $\deltaKS(u_B)$ is the KS-integral of the function
$G_B\colon \bN\times\BX\to\R$, defined by
$G_B(\mu,s,x):=u_B(s)\I\{x\le \varphi(s,\mu)\}$.
We are interested in the normal approximation of
$\deltaKS(u_B)$ for $B$ of growing volume.

\begin{remark}\label{rpredict}
  Assume that $d=1$ and that $\varphi$ is \emph{predictable}, that is,
  $\varphi(t,\mu)=\varphi(t,\mu_{t-})$, where $\mu_{t-}$ is the
  restriction of $\mu\in\bN$ to $(-\infty,t)\times\R_+$.  Then, under
  suitable integrability assumptions (satisfied under our assumptions below)
  $\big(\xi([0,t])-\int^t_0 \varphi(s,\eta)\,ds\big)_{t\ge 0}$ is a
  martingale with respect to the filtration
  $(\sigma(\eta_{(-\infty,t]\times\R_+}))_{t\ge 0}$; see e.g.\
  \cite{LastBrandt95}. Therefore, $(\varphi(t,\cdot))_{t\ge 0}$
  is a \emph{stochastic intensity} of $\xi$ (on $\R_+$) with respect
  to this filtration.
Take $B=[0,T]$ for some $T>0$ and write $u_T:=u_B$.  Then
  $(\deltaKS(u_T))_{T\ge 0}$ is a martingale.  Theorem~3.1 from
  \cite{Torrisi16} provides a quantitative central limit theorem in
  the Wasserstein distance for $\deltaKS(u_T)$. Below we derive a
  similar result using our tools, not only for the Wasserstein but
  also for the Kolmogorov distance. It should be noted that 
  predictability and martingale properties are of no relevance for our
  approach.  All what matters is that $\deltaKS(u_B)$ is a KS-integral
  with respect to the Poisson process $\eta$.
\end{remark}


Before stating some assumptions on $\varphi$, we introduce some useful
terminology.  A mapping $Z$ from $\bN$ to the Borel sets of $\BX$ is
called \emph{graph-measurable} if
$(\mu,s,x)\mapsto \I\{(s,x)\in Z(\mu)\}$ is a measurable
mapping. Given such a mapping, we define a whole family of $Z_t$,
$t\in\R^d$, of such mappings by setting
$$
Z_t(\mu):=Z(\theta_t\mu)+t,
$$
where $\theta_t\mu:=\int \I\{(r-t,z)\in\cdot\}\,\mu(d(r,z))$ is
the shift of $\mu$ by $t$ in the first coordinate, 
and $A+t:=\{(s+t,x):(s,x)\in A\}$ for any $A\subset\R^d\times\R_+$.

We assume that there exists a graph-measurable $Z$ such that
\begin{align}\label{eaa2}
  \varphi(t,\mu+\mu')=\varphi(t,(\mu+\mu')_{Z_t(\mu)}),
  \quad (t,\mu,\mu')\in\R^d\times\bN\times\bN,\,
  \mu'(\BX)\le 3.
\end{align}
Here, we denote by $\nu_A$ the restriction of a measure $\nu$ to a Borel set $A$ of $\BX$. 
Next, we assume that there exists a measurable mapping
$Y\colon\bN\to\R_+$ such that
\begin{align}\label{eaa3}
  \varphi(t,\mu+\mu')\le Y(\theta_t\mu),
  \quad (t,\mu,\mu')\in\R^d\times\bN\times\bN,\, \mu'(\BX)\le 3.
\end{align}
We let $Y_t(\eta)=Y(\theta_t\eta)$ for $t\in\R^d$. As in the
rest of the paper we write $Z_t$, $Y_t$ and $\varphi_t$ instead of
$Z_t(\eta)$, $Y_t(\eta)$ and $\varphi_t(\eta)$ for $t\in\R^d$.
Finally, we need the following integrability assumptions:
\begin{align}\label{eaa5}
&\int_{\R^d} \big(\BE \lambda(Z_0\cap Z_s)^4\big)^{1/4}\,ds<\infty,\\
\label{eaa6}
&\int_{\R_+}\int_{\R^d} \BP((s,x)\in Z_0)^{1/4}\,ds\,dx<\infty,\\
\label{eaa7}
&\int_{\R_+}\int_{\R_+}\int_{\R^d} \BP((s,x)\in Z_0,(0,y)\in Z_s)^{1/3}\,ds\,dx\,dy<\infty,\\
\label{eaa8}
&\BE Y_0^4<\infty.
\end{align}
It follows from Fubini's theorem, H\"older's inequality and
\eqref{eaa6} that $\BE \lambda(Z_0)^4<\infty$.

Assumptions \eqref{eaa3} and \eqref{eaa8} 
justify that $\deltaKS(u_B)$ is defined pathwise if $u$ is bounded.
Moreover, we will see below that 
our assumptions imply that \eqref{eH2} and
\eqref{eH3} hold. Therefore, $G_B$ is in the domain of 
the KS-integral.

Next we illustrate \eqref{eaa2} and \eqref{eaa5}--\eqref{eaa7} with a simple example.
Further examples will be discussed later in the section.

\begin{example}\label{exZd1}
Assume that $d=1$.
A simple (deterministic) choice of the sets $Z_t$ is $Z_t:=[t-h,t)\times C$,
where $h>0$ and $C\subset \R_+$ is  a bounded Borel set. 
If we assume that $\varphi(t,\mu)=\varphi(t,\mu_{Z_t})$ for all $(t,\mu)$,
then \eqref{eaa2} holds, while \eqref{eaa5}--\eqref{eaa7} are trivially true.
To discuss another, less trivial, choice we fix
another Borel set $C'\subset \R_+$ with $0<\lambda_+(C')<\infty$ and $n\in\N$. For $\mu\in \bN$ and
$t\in\R$ let $T^t_n(\mu)$ denote the $n$-th point of $\mu(\cdot\times C')$ strictly before $t\in\R$. 
Define $Z_t(\mu):=[T^t_n(\mu),t)\times C$. 
Then $Z_t(\mu)=Z(\theta_t\mu)+t$, and we have
\begin{align}
Z_t(\mu+\mu')=Z_t((\mu+\mu')_{Z_t(\mu)})\;\text{and}\; Z_t(\mu+\mu')\subset Z_t(\mu),
\quad (t,\mu,\mu')\in\R_+\times\bN\times \bN.
\end{align}
Assuming again that $\varphi(t,\mu)=\varphi(t,\mu_{Z_t})$, we easily obtain
\eqref{eaa2}. It is straightforward to check that \eqref{eaa5}--\eqref{eaa7} hold.
\end{example}

For the normal approximation of $\deltaKS(u_B)$ we have the following result.

\begin{theorem}\label{tembedWK} Let
$\varphi\colon\R^d\times\bN\to[0,\infty]$ be measurable,
and let $Z$ be a graph measurable mapping from $\bN$ to the Borel sets of $\BX$.
Assume that \eqref{eaa2}--\eqref{eaa8} are satisfied.
Let $u\colon\R^d\to\R$ be measurable and bounded, and let $B\subset\R^d$ be
  a Borel set with $\lambda_d(B)<\infty$. Finally, assume 
that $\sigma^2_B:=\BV(\deltaKS(u_B))>0$.
  Then there exists a constant $c>0$, not depending on $B$, such that
  \begin{align}\label{eembeddingWK}
    \max\big\{d_W\big(\sigma_B^{-1}\deltaKS(u_B),N\big),
    d_K\big(\sigma_B^{-1}\deltaKS(u_B),N\big)\big\}
    \le c \lambda_d(B)^{1/2}\sigma^ {-2}_B+c \lambda_d(B)\sigma^ {-3}_B.
  \end{align}
\end{theorem}
\begin{proof} 
  We apply Theorem~\ref{tmain} with $G_B/\sigma_B$ in place of
  $G$. For notational simplicity we omit the subscript $B$ of $G_B$.
  We need to bound the terms $T_i$ for $i\in\{1,\ldots,9\}$.  The
  assumptions of Theorem~\ref{tmain} are checked at the end of the
  proof. For simplicity, assume that $|u|$ is bounded by $1$. 
The value of a constant $c$ might change from line to line.  We often
write $D_{s,x}$ instead of $D_{(s,x)}$.

The term $T_3':= \sigma^{3}_B T_3$ satisfies 
\begin{align*}
T_3'\le \BE \int_B \varphi_s\,ds
\le c \lambda_d(B),
\end{align*}
where the second inequality follows from assumptions \eqref{eaa3} and
\eqref{eaa8}. Here and later we often use that $\theta_s\eta$ and
$\eta$ have the same distribution for each $s\in\R^d$, whence $Y_s$
has the same distribution for all $s\in\R^d$ and the same
holds for $\lambda(Z_s)$.

We deduce from \eqref{eaa2} that, for $(s,x)\in \BX$,
$(t,y)\notin Z_s$ and $\nu\in\mathbf{N}$ with $\nu(\BX)\leq 2$,
\begin{align*}
\I\{ x\leq \varphi_s(\eta+\nu+\delta_{(t,y)}) \} & = \I\{ x\leq \varphi_s((\eta+\nu+\delta_{(t,y)})_{Z_s}) \} \\
& = \I\{ x\leq \varphi_s((\eta+\nu)_{Z_s}) \} = \I\{ x\leq \varphi_s(\eta+\nu) \},
\end{align*}
whence the first three difference operators of
$\I\{ x\leq \varphi_s \}$ vanish if one of the additional points is
outside of $Z_s$. From \eqref{eaa3} we see that
$\I\{ x\leq \varphi_s \}$ and its first three difference operators
become zero if $x> Y_s$. In the following, these
observations are frequently used to bound difference operators in
terms of indicator functions.
 
First we consider $T_1$. Writing the square of the inner integral
as a double integral,
we have
\begin{align*}
T'_1:=\sigma^4_BT^2_1\le
\BE\int \I\{s,r\in B\} |D_{t,y}\I\{x\le \varphi_s\}||D_{t,y}\I\{z\le \varphi_r\}|\,d(s,x,t,y,r,z).
\end{align*}
By the discussed behaviour of the difference operators, 
\begin{align*}
T'_1&\le c\,\BE\int \I\{s,r\in B\} \I\{(t,y)\in Z_s\cap Z_r\}\I\{x\le Y_s,z\le Y_r\}|\,d(s,x,t,y,r,z)\\
&= c\,\BE\int_{B^2} \lambda(Z_s\cap Z_r)Y_sY_r\,d(s,r)\\
& \leq c\,\int_{B^2}  \big(\BE\lambda(Z_s\cap Z_r)^3\big)^{1/3}\big(\BE Y^3_s\big)^{1/3}
\big(\BE Y_r^3\big)^{1/3}\,d(s,r),
\end{align*}
where we have used H\"older's inequality. 
By \eqref{eaa8}, $\BE Y_s^3=\BE Y_r^3=\BE Y_0^3<\infty$.
Moreover,
\begin{align*}
\BE\lambda(Z_s\cap Z_r)^3& 
=\BE\lambda((Z(\theta_s\eta)+s)\cap (Z(\theta_r\eta)+r))^3\\
&=\BE\lambda((Z(\theta_{s-r}\eta)+s-r)\cap Z(\eta))^3.
\end{align*}
Therefore,
\begin{align*}
T'_1&\le c\,\int\I\{s\in\R^d,r\in B\}
\big(\BE \lambda((Z(\theta_{s}\eta)+s)\cap Z(\eta))^3\big)^{1/3}\,d(s,r)\\
&=c\lambda_d(B) \int_{\R^d}\big(\BE \lambda(Z_s\cap Z_0)^3\big)^{1/3}\,ds
\le c\lambda_d(B),
\end{align*}
where we have used assumption \eqref{eaa5}
(and the monotonicity of $L_p$-norms).
Hence, $T_1\le c \lambda_d(B)^{1/2}\sigma^ {-2}_B$, as required by
\eqref{eembeddingWK}.

For the term $T_2$, we have
\begin{align*}
T'_2:=\sigma^4_BT^2_2\le
\BE\int\bigg(\int \I\{s,t\in B\}|D_{r,z}(D_{s,x}\I\{y\le \varphi_t\} D_{t,y}\I\{x\le \varphi_s\})|
\,d(s,x,t,y)\bigg)^2d(r,z).
\end{align*}
The inner integrand does only contribute
if $(s,x)\in Z_t$, $(t,y)\in Z_s$, and $(r,z)\in Z_t$ or $(r,z)\in Z_s$.
Since the last two cases are symmetric, $T'_2$ can be bounded by
\begin{align*}
c\,\BE\int\bigg(\int \I\{t\in B\} \I\{(r,z)\in Z_t,(s,x)\in Z_t,(t,y)\in Z_s\}
\,d(s,x,t,y)\bigg)^2d(r,z).
\end{align*}
By Fubini's theorem, 
\begin{align*}
T'_{2}
&\le c\,\BE\int \I\{t,t'\in B\} \lambda(Z_t\cap Z_{t'})
\I\{(s,x)\in Z_t,(t,y)\in Z_s,(s',x')\in Z_{t'},(t',y')\in Z_{s'}\} \\
& \qquad\qquad\qquad\times \,d(s,x,t,y,s',x',t',y')\\
&\le c\int \I\{t,t'\in B\}  \big(\BE\lambda(Z_t\cap Z_{t'})^3\big)^{1/3}
\,\BP((s,x)\in Z_t,(t,y)\in Z_s)^{1/3}\\
&\qquad\qquad\qquad\times \BP((s',x')\in Z_{t'},(t',y')\in Z_{s'})^{1/3}
\,d(s,x,t,y,s',x',t',y').
\end{align*}
By definition of $Z_t$ and $Z_s$ and the distributional invariance of $\eta$,
\begin{align*}
\BP((s,x)\in Z_t,(t,y)\in Z_s)
&=\BP((s-t,x)\in Z(\theta_t\eta),(t-s,y)\in Z(\theta_s\eta))\\
&=\BP((s-t,x)\in Z(\eta),(t-s,y)\in Z(\theta_{s-t}\eta)).
\end{align*}
Changing variables yields that
\begin{align*}
T'_{2} &\le cb^2\,\int_{B^2} (\BE\lambda(Z_t\cap Z_{t'})^3)^{1/3} \,d(t,t'),
\end{align*}
where
\begin{align*}
b:=\int\BP((s,x)\in Z(\eta),(-s,y)\in Z(\theta_{s}\eta))^{1/3}\,d(s,x,y).
\end{align*}
Since  
\begin{align*}
\BP((s,x)\in Z(\eta),(-s,y)\in Z(\theta_{s}\eta))
=\BP((s,x)\in Z_0,(0,y)\in Z_s),
\end{align*}
we obtain from assumption \eqref{eaa7} that $b<\infty$. Hence,
\begin{align*}
T'_{2} &\le c\,\int_{B^2} \big(\BE \lambda(Z_t\cap Z_{t'})^{3}\big)^{1/3} \,d(t,t')
=c \int_{B^2} \big(\BE \lambda(Z_{t-t'}\cap Z_{0})^3\big)^{1/3} \,d(t,t')\le c\lambda_d(B),
\end{align*}
where we have used assumption \eqref{eaa5}.

Each of the summands in the term $T'_4:=\sigma^3_BT_4$ includes the
factor $D_{s,x}\I\{ y\leq \varphi_t \}$, so that
\begin{align*}
T_4' & \leq c\, \BE \int \I\{t\in B\} \I\{(s,x)\in Z_t\} \I\{y \leq Y_t\} \, d(s,x,t,y) = c \BE \int_{B} \lambda(Z_t) Y_t \, dt \\
&  \leq c \int_{B} \big(\BE \lambda(Z_t)^2 \big)^{1/2}  \big(\BE Y_t^2 \big)^{1/2} \, dt = c \big(\BE \lambda(Z_0)^2\big)^{1/2}  \big(\BE Y_0^2 \big)^{1/2} \lambda_d(B).
\end{align*}

For $T'_5:=\sigma^3_BT_5$, we have
$$
T_5' \leq c\, \BE \int \I\{r\in B\} \I\{(s,x) \in Z_t\} \I\{ (t,y) \in Z_r \} \I\{z\leq Y_r\} \, d(s,x,t,y,r,z),
$$
where in the second term of $T_5$ we renamed $x$ as $y$ and vice
versa. This leads to the upper bound
\begin{align*}
T_5' & \leq c\, \BE \int \I\{r\in B\} \I\{ (t,y) \in Z_r \} \lambda(Z_t) Y_r \, d(t,y,r) \\
& \leq c \int \I\{r\in B\} \BP((t,y) \in Z_r)^{1/3} \big(\BE\lambda(Z_t)^3\big)^{1/3} \big(\BE Y_r^3\big)^{1/3} \, d(t,y,r) \\
& = c \big(\BE\lambda(Z_0)^3\big)^{1/3} \big(\BE Y_0^3\big)^{1/3} \int \BP((t,y) \in Z_0)^{1/3} \, d(t,y) \lambda_d(B).
\end{align*}

We can rewrite $T'_6:=\sigma^4_B T_6^2$ as sum of $T_{6,1}'$ and $T_{6,2}'$ with
\begin{align*}
T_{6,1}' & \leq c\, \BE \int \bigg( \int \I\{t\in B\} \I\{ y\leq Y_t \} \I\{(s,x)\in Z_t\} \, d(s,x) \bigg)^2  \, d(t,y) \\
& \leq c\, \BE \int \I\{t\in B\} \I\{ y\leq Y_t \} \lambda(Z_t)^2 \, d(t,y) = c \BE \int_B Y_t \lambda(Z_t)^2 \, dt \\
& \leq c \int_B \big(\BE Y_t^3\big)^{1/3} \big(\BE \lambda(Z_t)^3\big)^{2/3} \, dt = c \big(\BE Y_0^3\big)^{1/3} \big(\BE \lambda(Z_0)^3\big)^{2/3} \lambda_d(B)
\end{align*}
and
\begin{align*}
T_{6,2}' & \leq c\, \BE \int \bigg( \int \I\{t\in B\} \I\{ y\leq Y_t \} \I\{(s,x)\in Z_t\} \I\{ (r,z)\in Z_s\cup Z_t \} \, d(s,x)  \bigg)^2  \, d(t,y,r,z) \\
& = c\, \BE \int \I\{t\in B\} \I\{ y\leq Y_t \} \I\{(s,x)\in Z_t\} \I\{ (r,z)\in Z_s\cup Z_t \} \\
& \hspace{4cm} \times \I\{(s',x')\in Z_t\} \I\{ (r,z)\in Z_{s'}\cup Z_t \} \, d(s,x,s',x',t,y,r,z) \\
& = c\, \BE \int \I\{t\in B\} Y_t \I\{(s,x),(s',x')\in Z_t\} \lambda(( Z_s\cup Z_t)\cap (Z_{s'}\cup Z_t)) \, d(s,x,s',x',t) \\
& \leq c \int \I\{t\in B\}  \BP((s,x)\in Z_t)^{1/4} \BP((s',x')\in Z_t)^{1/4} \big(\BE Y_t^4\big)^{1/4} \\
& \hspace{4cm} \times \big( \big(\BE \lambda(Z_s)^4\big)^{1/4}+\big(\BE \lambda(Z_t)^4\big)^{1/4}\big) \, d(s,x,s',x',t) \\
& = 2c \big(\BE Y_0^4\big)^{1/4} \big(\BE \lambda(Z_0)^4\big)^{1/4} \bigg( \int \BP((s,x)\in Z_0)^{1/4} \, d(s,x) \bigg)^2 \lambda_d(B).
\end{align*}

For $T'_7:=\sigma^4_B T_7^2$, the first term can be bounded as $T_3'$, while the second term is bounded by
\begin{equation}\label{eqn:Integral_2}
c\, \BE \int \I\{t\in B\} \I\{ (s,x)\in Z_t\} \I\{y\leq Y_t\} \, d(s,x,t,y),
\end{equation}
which we treated above in order to control $T_4'$.

We can decompose $T_8':=\sigma_B^4 T_8^2$ into two terms $T_{8,1}'$ and $T_{8,2}'$, where $T_{8,1}'$ can be bounded as $T_{6,1}'$. Since the product of two difference operators in $T_{8,2}'$ is bounded by the sum of the squared difference operators, $T_{8,2}'$ can be controlled as $T_{6,2}'$.

Note that $T_9':= \sigma_B^4 T_9^2$ can be written as a sum of
three terms $T_{9,1}',T_{9,2}',T_{9,3}'$, where $T_{9,i}'$ is an
integral with respect to $i$ points for $i\in\{1,2,3\}$. The term
$T_{9,1}'$ can be bounded by \eqref{eqn:Integral_2}, while
\begin{align*}
T_{9,2}' & \leq c\, \BE \int \I\{t\in B\} \I\{ y \leq Y_t \} \I\{(s,x)\in Z_t\} \I\{ (r,z) \in Z_s \cup Z_t \} \, d(s,x,t,y,r,z) \\
& \leq c\, \BE \int \I\{t\in B\} Y_t \I\{(s,x)\in Z_t\} (\lambda(Z_s) + \lambda(Z_t)) \, d(s,x,t) \\
& \leq c \int \I\{t\in B\} \big(\BE Y_t^3\big)^{1/3} \BP((s,x)\in Z_t)^{1/3} \big(\big(\BE\lambda(Z_s)^3\big)^{1/3} + \big(\BE \lambda(Z_t)^3\big)^{1/3}\big) \, d(s,x,t) \\
& \leq 2c \big(\BE Y_0^3 \big)^{1/3} \big(\BE\lambda(Z_0)^3\big)^{1/3} \int \BP((s,x)\in Z_0)^{1/3}  \, d(s,x) \lambda_d(B).
\end{align*}
For $T_{9,3}'$ we deduce the bound
\begin{align*}
T_{9,3}' & \leq c\, \BE\int \I\{t\in B\}\I\{y\le Y_t\} \I((s,x)\in Z_t) \I\{(s',x'),(r,z)\in Z_s\cup Z_t\}\, d(s,x,t,y,r,z,s',x') \\
& \leq c\, \BE\int \I\{t\in B\} Y_t \I((s,x)\in Z_t)(\lambda(Z_s)+\lambda(Z_t))^2 \, d(s,x,t),
\end{align*}
which can be treated similarly as in the computation for $T_{9,2}'$ but with the power $4$.

Finally, we check the assumptions of Theorem \ref{tmain}. The expression in \eqref{eH2} can be treated as $T_3'$, while \eqref{eH3}, \eqref{eH_mixed} and \eqref{eH_mixed_2} can be bounded as $T_4'$. Similarly, we can verify \eqref{eH4}, \eqref{eH5} and \eqref{eH_mixed_3} by using the computations for $T_{9,2}'$, $T_{9,3}'$ and $T_{6,2}'$, respectively. 
\end{proof}

\begin{remark}\label{rpredict3}
Theorem \ref{tembedWK} can be used to establish central limit theorems.
Consider, for instance, the setting of Remark~\ref{rpredict}.
Two possible choices of $Z_t$ are provided in Example~\ref{exZd1}.
Since $\varphi$ is assumed to be predictable in Remark~\ref{rpredict},
the cyclic condition \eqref{eq:3} is satisfied and \eqref{evariance} simplifies to 
\begin{align*}
\sigma_T^2:=\BV(\deltaKS(u_T))=\int^T_0 u(t)^2\, \BE\varphi(t,\eta)\,dt.
\end{align*}
It is natural to assume that $\sigma_T^2\ge c T$ for some $c>0$ and
all sufficiently large $T$. If, additionally, the assumptions of Theorem \ref{tembedWK}
are satisfied, then \eqref{eembeddingWK} shows that
\begin{align*}
\max\big\{d_W\big(\sigma_T^{-1}\deltaKS(u_T),N\big),d_K\big(\sigma_T^{-1}\deltaKS(u_T),N\big)\big\}
\le c' T^{-1/2}
\end{align*}
for some $c'>0$ and all sufficiently large $T$.
It does not seem to be possible to derive  
the Wasserstein part of this bound from \cite[Theorem 3.1]{Torrisi16};
see also \cite[Remark 3.8]{HHKR21}. The reason is that the third
term on the right-hand side of \cite[(3.9)]{Torrisi16} does not have the
appropriate order.
\end{remark}

\begin{example}
Let $h\colon\R^d\to\R_+$ be a measurable
function satisfying $\int (h(s)+h(s)^2)\,ds<\infty$. Define
$Z:=\{(s,x)\in\R^d\times\R_+: x\le h(s)\}$ and
$Z_t:=Z+t$, $t\in\R^d$.
We interpret $Z$ and $Z_t$ as constant mappings on $\bN$ 
and check that \eqref{eaa5}-\eqref{eaa7} are satisfied.
For \eqref{eaa5} we note that
\begin{align*}
\int \lambda(Z_0\cap Z_s)\,ds
&=\int \I\{y\le h(t),y\le h(t-s)\}\,d(t,y,s)\\ 
&=\int \I\{y\le h(t),y\le h(s)\}\,d(t,y,s)
=\int \bigg(\int \I\{y\le h(s)\}\,ds\bigg)^2\,dy. 
\end{align*}
Since $h$ is square integrable, we have
$\int \I\{y\le h(s)\}\,ds\le c y^{-2}$ for some $c>0$,
so that the above integral is finite.
Relation \eqref{eaa6} follows at once from the integrability
of $h$, while the left-hand side of \eqref{eaa7}
is bounded by $\int h(s)^2\,ds$.

Assume now that the function $\varphi$ satisfies
\begin{align*}
\varphi(t,\mu)=\varphi(t,\mu_{Z_t}),\quad (t,\mu)\in\R^d\times\bN.
\end{align*}
Then \eqref{eaa2} holds. Assumptions \eqref{eaa3} and \eqref{eaa8}
depend on the choice of $\varphi$. They are satisfied, for instance,
if $\varphi(t,\cdot)$ is a polynomial or exponential function of
$\mu(Z_t)$.

Assume that $u$ and $\BE \varphi(\cdot,\eta_{Z})$ have a lower bound
$c>0$ and that $\varphi(s,\cdot)$ is for all $s\in\R^d$ either increasing
or decreasing when adding a point.
Then Theorem \ref{tembedWK}
yields a (quantitative) central limit theorem for $\lambda_d(B)\to\infty$.
To this end,
we need to find a lower bound for $\sigma^2_B$, given
by \eqref{evariance}. In our case the first term on the right-hand side
of \eqref{evariance} equals
\begin{align*}
\BE \int \I\{s\in B\}u(s)^2\I\{x\le \varphi(s,\eta_{Z_s})\}\,d(s,x)
\end{align*}
and has the lower bound
\begin{align*}
c^2\,\int \I\{s\in B\}\BE \varphi(s,\eta_{Z_s})\,ds\ge c^3\lambda_d(B).
\end{align*}
The second term is given by
\begin{align*}
\BE \int \I\{s,t\in B\}u(s)u(t) D_{t,y}\I\{x\le\varphi(s,\eta)\}D_{s,x}\I\{y\le\varphi(t,\eta)\}\,d(s,x,t,y).
\end{align*}
By the monotonicity assumption on $\varphi$ and $u\geq c$, this is non-negative.
\end{example}

\begin{example}
For a point configuration $\mu\in\mathbf{N}$ and $w\in\BX$ the Voronoi cell of $w$ is given by
$$
V(w,\mu) := \{v\in \BX: \|w-v\| \leq \|w'-v\| \text{ for all } w'\in\mu \},
$$
i.e., $V(w,\mu)$ is the set of all points in $\BX$ such that no point
of $\mu$ is closer than $w$. The cells $(V(w,\mu))_{w\in\mu}$ have
disjoint interiors and form a tessellation of $\BX$, the so-called
Voronoi tessellation, which is an often studied model from stochastic
geometry (see e.g.\ \cite[Section 10.2]{SW}). From the
Poisson--Voronoi tessellation (i.e., the Voronoi tessellation with
respect to $\eta$) we construct the point process
\begin{equation}\label{eqn:xi_PV}
\xi:= \int \mathbf{1}\{s\in \cdot\} \mathbf{1}\{ V((s,x),\eta)\cap (\R^d\times\{0\})\neq \varnothing \} \, \eta(d(s,x)).
\end{equation}
This point process has the following geometric interpretation. We take
all cells of the Poisson--Voronoi tessellation that intersect
$\R^d\times\{0\}$, which one can think of as the lowest layer
of the Poisson--Voronoi tessellation, and the first coordinates of
their nuclei are the points of $\xi$. The points of $\xi$ build the
projection of a one-sided version of the Markov path considered
in \cite{BTZ00}.

First we check that $\xi$ can be represented as in \eqref{eqn:xi}. For
$s\in\R^d$, $x_1,x_2\in\R_+$ with $x_1<x_2$ and
$\mu\in\mathbf{N}$ we have
\begin{equation}\label{eqn:monotonicity_intersection}
  V((s,x_1),\mu) \cap (\R^d\times\{0\}) \supset V((s,x_2),\mu) \cap (\mathbb{R}^d\times\{0\}).
\end{equation}
If $V((s,0),\mu)$ is bounded, which is for $\BP_\eta$-a.e.\ $\mu$ the case, 
there exists a unique $x_0\in\R_+$ such that
$V((s,x_0),\mu) \cap (\R^d\times\{0\})$ is exactly a single
point. This allows us to rewrite $\xi$ as
$$
\xi= \int \I\{s\in \cdot\} \I\{ x \leq \varphi(s,\eta-\delta_{(s,x)}) \} \, \eta(d(s,x))
$$
with
$$
\varphi(s,\mu):=\sup\{x\in\R_+: V((s,x),\mu) \cap (\R^d\times\{0\})\neq \varnothing \}.
$$

For $s\in\R^d$ and $\mu\in\mathbf{N}$ let
$$
R(s,\mu):=\sup\{\|(s,0)-v\|: v\in V((s,0),\mu)\},
$$
which is the maximal distance from $(s,0)$ to a point of its Voronoi
cell. Note that $V((s,0),\mu)$ is completely determined by the points
of $\mu$ in $B((s,0),2R(s,\mu))$, the closed ball in $\BX$ with radius
$2R(s,\mu)$ around $(s,0)$. Indeed, the centres of all neighbouring
cells to the Voronoi cell
of $(s,0)$ are within this ball and all other
points of $\eta$ outside are too far
away to affect the cell. If we consider
$V((s,x),\mu)\cap (\R^d\times\{0\})$ as a function of $x$, for
increasing $x$ the sets $V((s,x),\mu)\cap (\R^d\times\{0\})$
are not increasing (see \eqref{eqn:monotonicity_intersection}) and
$(V((s,0),\mu)\cap (\R^d\times\{0\})) \setminus
(V((s,x),\mu)\cap (\R^d\times\{0\}))$ is divided among the
neighbouring cells of $V((s,0),\mu)$. This implies that
$V((s,x),\mu)\cap (\R^d\times\{0\})$ is also completely
determined by the points in $B(s,2R((s,0),\mu))$. Hence, we can
conclude that
$$
\varphi(s,\mu) = \varphi(s,\mu_{B((s,0),2R(s,\mu))}).
$$
Since this identity is still valid if we restrict $\mu$ to a larger
set on the right-hand side and $R$ is non-increasing with respect to
the point configuration, we obtain
$$
\varphi(s,\mu+\mu') = \varphi(s, (\mu+\mu')_{B((s,0),2R(s,\mu+\mu'))} ) = \varphi(s, (\mu+\mu')_{B((s,0),2R(s,\mu))})
$$
for all $\mu'\in\mathbf{N}$ with $\mu'(\BX)\leq 3$, which is \eqref{eaa2} with $Z_s=B((s,0),2R(s,\mu))$. Since for each point of $V((s,0),\mu)\cap (\R^d\times\{0\})$, there exists a point of $\mu$ different from $(s,0)$ which is at most $2R(s,\mu)$ away, we obtain
$$
\varphi(s,\mu+\mu') \leq \varphi(s,\mu) \leq 2R(s,\mu), 
$$
which is \eqref{eaa3}.

Note that for any $s\in\R^d$ one can partition $\BX$ into
  finitely many cones $\mathcal{C}_1,\hdots,\mathcal{C}_m$ with apex
  $(s,0)$ such that
$$
\max_{i\in\{1,\hdots,m\}} \inf_{y\in \mu \cap \mathcal{C}_i} \|y - (s,0)\| \geq R(s,\mu)
$$
for all $\mu\in\mathbf{N}$ (see
e.g.\ \cite[Subsection~6.3]{P07}). Hence, there
exist constants $C,c>0$ such that
$$
\mathbb{P}(R(s,\eta)\geq u) \leq C \exp(-c u^{d+1})
$$
for all $u\geq 0$ and $s\in\R^d$. 
Using this exponential decay it is easy to verify
\eqref{eaa5}--\eqref{eaa8}. Relations \eqref{eaa6} and \eqref{eaa8} 
are obvious.  To see \eqref{eaa5}, we can use the bound
\begin{multline*}
\lambda(B((0,0),2R(0,\eta))\cap B((s,0),2R(s,\eta)))^4\\
\le \I\{2R(0,\eta)>\|s\|/2\}\lambda(B((0,0),2R(0,\eta)))^4\\
+\I\{2R(s,\eta)>\|s\|/2\}\lambda(B((s,0),2R(s,\eta)))^4.
\end{multline*}
For \eqref{eaa7} we can bound
$\BP((s,x)\in B((0,0),2R(0,\eta)), (0,y)\in B((s,0),2R(s,\eta)))$
by the Cauchy--Schwarz inequality 
and then bound the resulting integral.
This yields that the conclusions of
Theorem~\ref{tembedWK} hold for the point process $\xi$ from
\eqref{eqn:xi_PV}.

Since $\varphi$ is non-increasing with respect to additional points,
one can argue as in the previous example to see that there is a lower
bound for the variance of order $\lambda_d(B)$ if $u>c_0$ for some
$c_0>0$. This yields a (quantitative) central limit theorem as
$\lambda_d(B)\to\infty$.
\end{example}

\section{Functionals generated by a partial order}
\label{sec:integrals-functions}

In this section we return to the setting of a general $\sigma$-finite
measure space $(\BX,\cX,\lambda)$.  In many situations, the functional
$G_x$ can be written as $G_x(\mu)=f(x)H_x(\mu)$, where
$f\in L^2(\lambda)$ and the functional $H_x(\mu)$ is measurable in
both arguments, takes values in $\{0,1\}$ and can be decomposed as
\begin{equation}
  \label{eq:2}
  H_x(\mu)=\prod_{y\in\mu} H_x(\Dirac_y).
\end{equation}
Write shortly $H_x(y)$ instead of $H_x(\Dirac_y)$, and denote
$\bH_x(y):=1-H_x(y)$. A generic way to construct such functionals is
to consider a strict partial order $\prec$ on $\BX$ and to set
$H_x(y):=1-\I\{y\prec x\}$. The set of points $x\in\eta$ such that
$H_x(\eta)=1$ is called the set of \emph{Pareto optimal} points with
respect to the chosen partial order, i.e., $x\in\eta$ is Pareto
optimal if there exists no $y\in\eta$ such that $y\prec x$. For
$x\notin\eta$, we have $H_x(\eta)=1$ if $x$ is Pareto optimal in
$\eta+\Dirac_x$.  If $\deltaKS(G)$ can be defined pathwise as in
\eqref{e1.1}, then it equals the sum of the values of $f$ over Pareto
optimal points centred by the integral of $f$ over the set of $x$ such
that $H_x(\eta)=1$.  As shown in \cite{LM21}, such examples naturally
arise in statistical applications.

It is easy to see by induction that 
\begin{equation}
  \label{eq:7}
  D^m_{z_1,\dots,z_m}G_x(\mu)=(-1)^m f(x)  H_x(\mu) \prod_{i=1}^m
  \bH_x(z_i). 
\end{equation}
In particular,
\begin{equation}
  \label{eq:6}
  D_zG_x(\mu)=-f(x)H_x(\mu)\bH_x(z).
\end{equation}
By construction, $H_y(\eta)=1$ and $\bH_y(x)=1$ yield that
$H_x(\eta)=1$, which can be expressed as
\begin{equation}\label{eqn:H3H2}
  H_x(\eta)H_y(\eta)\bH_y(x)=H_y(\eta)\bH_y(x),
\end{equation}
so that
\begin{equation}
  \label{eq:9}
  G_xD_xG_y =f(x) D_xG_y.
\end{equation}

The asymmetry property of the strict partial order implies that
$\bH_x(y)\bH_y(x)=0$ for all $x,y\in\BX$. Hence, the functional $G$
satisfies the cyclic condition \eqref{eq:3}. Thus, the second term on
the right-hand side of \eqref{evariance} vanishes. If \eqref{eH2}
and \eqref{eH3} are satisfied, it follows from
\cite[Proposition~2.3]{las:pec:sch16} that the KS-integral $\deltaKS(G)$
of $G$ is well defined and
\begin{equation}
  \label{eq:10}
  \BE \deltaKS(G)^2=\BE \int f(x)^2 H_x(\eta)\,\lambda(dx).
\end{equation}

In addition, property \eqref{eq:2} leads to a considerable
simplification of the terms arising in the bounds in
Corollary~\ref{cor:cyclic}.  Write $H_x$ as a shorthand for $H_x(\eta)$,
denote
\begin{displaymath}
  h_i(y):=\int f(x)^i\bH_y(x)\,\lambda(dx), \quad i=0,1,2,
\end{displaymath}
and
\begin{displaymath}
  \tilde{h}(y):=\int |f(x)|\bH_y(x)\,\lambda(dx).
\end{displaymath}

\begin{proposition}\label{prop:7.1}
  Assume that $G_x(\mu)=f(x)H_x(\mu)$, where $f\in L^2(\lambda)$ and the functional $H$ is 
determined by \eqref{eq:2} from a strict
  partial order on $\BX$. Then the terms $T_2$ and $T_8$ defined before
  Theorem~\ref{tmain} vanish and the other terms satisfy
  \begin{align*}
    T_1&=\bigg(\int f(x)^2f(z)^2 \BE H_xH_z \bH_x(y)\bH_z(y)
           \,\lambda^3(d(x,y,z))\bigg)^{1/2},\\
    T_3&=\int |f(x)|^3 \BE H_x\,\lambda(dx),\\
    T_4&\leq \int \big(2 h_2(y)|f(y)|+3 \tilde{h}(y)f(y)^2\big)
         \BE H_y\,\lambda(dy),\\
    T_5&\leq 8 \int \tilde{h}(z)^2|f(z)|\BE H_z\,\lambda(dz),\\
    T_6&=\bigg(\int \big(f(y)h_1(y)\big)^2\big(1+h_0(y)\big)
         \BE H_y\,\lambda(dy)\bigg)^{1/2},\\
    T_7&= \Big(\int |f(x)|^4 \BE H_x \,\lambda(dx)\Big)^{1/2},\\
    T_9&= \bigg(\int f(y)^2\Big[3+3h_0(y)+2h_0(y)^2\Big]h_2(y)
       \BE H_y \,\lambda(dy)\bigg)^{1/2}.
  \end{align*}
	Suppose $\BV \deltaKS(G)>0$ and that \eqref{eH2}--\eqref{eH5} are satisfied. Then
	$$
	d_W\bigg(\frac{\deltaKS(G)}{\sqrt{\BV \deltaKS(G)}},N\bigg)\le \frac{T_1}{\BV \deltaKS(G)} + \frac{T_3+T_4+T_5}{\sqrt{\BV \deltaKS(G)}^3}.
	$$
	If, additionally, \eqref{eH_mixed}--\eqref{eH_mixed_3} are satisfied, then
	$$
	d_K\bigg(\frac{\deltaKS(G)}{\sqrt{\BV \deltaKS(G)}},N\bigg)\le \frac{T_1+T_6+2(T_7+T_9)}{\BV \deltaKS(G)}.
	$$
\end{proposition}
\begin{proof}
 
  The expression for $T_1$ follows from $G_x^2=f(x)G_x$ for
  $x\in\BX$ and \eqref{eq:6}, while $T_3$ results from the
  definition of $G_x$.  Now consider the further terms, appearing in
  Corollary~\ref{cor:cyclic}. 
We rely on \eqref{eq:7} with $m=2,3$, \eqref{eq:6}, and \eqref{eq:9} in the subsequent calculations. 
First,  \begin{align*}
    T_4&=\BE \int \Big(2f(x)^2 |f(y)| H_y\bH_y(x)\\
       &\qquad\qquad +|f(x)|f(y)^2
       H_y\bH_y(x)\big(2H_y+H_y\bH_y(x)\big)
       \Big)\,\lambda^2(d(x,y))\\
     &\leq \int \big(2f(x)^2 |f(y)|+3 |f(x)|f(y)^2\big)
       \BE H_y\bH_y(x) \,\lambda^2(d(x,y)),
  \end{align*}
  which yields the expression for $T_4$ in view of the definitions of the
  functions $h_2$ and $\tilde{h}$.  Next,
  \begin{align*}
    T_5&=\BE \int 2|f(x)f(y)f(z)|
       \big(H_z\bH_z(y)+H_z\bH_z(y)\bH_z(x)\big)\\
     &\qquad\qquad\qquad\qquad\times \big(H_y\bH_y(x)\bH_y(z)+ 2 H_y\bH_y(x)\big)
       \,\lambda^3(d(x,y,z)) \allowdisplaybreaks \\
     &= \BE \int 2|f(x)f(y)f(z)|
       H_z\bH_z(y)\big(1+\bH_z(x)\big)
       H_y\bH_y(x)\big(\bH_y(z)+2\big)
       \,\lambda^3(d(x,y,z)) \allowdisplaybreaks \\
     &\leq 8 \BE \int |f(x)f(y)f(z)|
       H_z\bH_z(y)H_y\bH_y(x) \,\lambda^3(d(x,y,z)) \allowdisplaybreaks \\
     &= 8 \int |f(x)f(y)f(z)|\BE H_z\bH_z(y)\bH_y(x)
       \,\lambda^3(d(x,y,z)),
\end{align*}
where we used the fact that $\bH_z(y) \bH_y(z)=0$ for all $y$ and
$z$ as well as \eqref{eqn:H3H2}. This yields the sought bound for $T_5$, taking into account that
$\bH_z(y)\bH_y(x)\leq \bH_z(y)\bH_z(x)$. 
Next, $T_6=(T_{6,1}+T_{6,2})^{1/2}$, where
\begin{align*}
  T_{6,1}&:=\BE\int \Big(\int
  f(x)f(y) H_y\bH_y(x) \,\lambda(dx)\Big)^2\,\lambda(dy)\\
         &=\int f(y)^2 \BE H_y
           \Big(\int f(x)\bH_y(x)\,\lambda(dx)\Big)^2\,\lambda(dy)
         =\int f(y)^2 h_1(y)^2 \BE H_y\,\lambda(dy)
\end{align*}
and 
\begin{align*}
  T_{6,2}&:=\BE \int \Big(\int
           f(x)f(y)H_y\bH_y(x)\bH_y(z)\,\lambda(dx)\Big)^2\,\lambda^2(d(y,z))\\
  &=\int f(y)^2 \BE H_y \bH_y(z) h_1(y)^2\,\lambda^2(d(y,z)).
\end{align*}
Hence, the expression for $T_6$ follows. The expression for $T_7$
follows directly from the definition of $G_x$. 
Finally, $T_9=(3T_{9,1}+3T_{9,2}+2T_{9,3})^{1/2}$, where
\begin{align*}
  T_{9,1}&:=\int f(x)^2 f(y)^2 \BE H_y\bH_y(x)\,\lambda^2(d(x,y)),\\
  T_{9,2}
    &:=\int f(x)^2f(y)^2 \BE H_y\bH_y(x)\bH_y(z)\,\lambda^3(d(x,y,z)),\\
  T_{9,3}&:=\int f(x)^2f(y)^2\BE H_y\bH_y(x)\bH_y(z)\bH_y(w)
  \,\lambda^4(d(x,y,z,w)).
\end{align*}
Thus, 
\begin{align*}
  T_9 &= \bigg(\int f(x)^2f(y)^2\Big[3+3h_0(y)+2h_0(y)^2]
       \BE H_y \bH_y(x)\,\lambda^2(d(x,y))\bigg)^{1/2},
\end{align*}
which yields the formula for $T_9$. The bounds for the normal approximation follow from Corollary \ref{cor:cyclic} and the normalisation by $\sqrt{\BV \deltaKS(G)}$.
\end{proof}

\begin{example}
  Let $\BX$ be the unit cube $[0,1]^d$ with the Lebesgue measure
  $\lambda$. For $x,y\in\BX$, write $y\prec x$ if $x\neq y$ and all
  components of $y$ are not greater than the corresponding components
  of $x$. Let $G_x(\mu)=H_x(\mu)$, with $H_x(\mu)$ given by
  \eqref{eq:2} and $\bH_x(y):=\I\{y\prec x\}$.

  Let $\eta_t$ be the Poisson process on $\BX$ of intensity
  $t\lambda$.  Then $G_x(\eta_t)=1$ means that none of the points
  $y\in\eta_t$ satisfies $y\prec x$, that is, none of the points from
  $\eta_t$ is smaller than $x$ in the coordinatewise order. In this
  case, $x$ is said to be a Pareto optimal point in
  $\eta_t+\Dirac_x$. Then $\deltaKS(G)$ equals the difference between
  the number of Pareto optimal points in $\eta_t$ and the volume of
  the complement of the set of points $x\in\BX$ such that $y\prec x$
  for at least one $y\in\eta_t$.

  For $x=(x_1,\dots,x_d)\in\BX$, denote $|x|:=x_1\cdots x_d$. Then
  $\BE H_x(\eta_t)=e^{-t|x|}$, and \eqref{eq:10} yields that the
  variance of $\deltaKS(G)$ is
  \begin{displaymath}
    \sigma^2_t:=t\int e^{-t|z|}\, \lambda(dz).
  \end{displaymath}
  It is shown in \cite{bai05} that the right-hand side is of order
  $\log^{d-1} t$ for large $t$. Note that the above formula gives also
  the expected number of Pareto optimal points.

  Quantitative limit theorems for the number of Pareto optimal points
  centred by subtracting the mean and scaled by the standard
  deviation were obtained in \cite{bhat:mol21}. Below we derive a
  variant of such result for the KS-integral, which
  involves a different stochastic centring.

Since $G_x(\eta)=f(x)H_x(\eta)$ with the function $f$
identically equal one and the measure $\lambda$ is finite, 
the integrability conditions \eqref{eH2}--\eqref{eH5}, 
and \eqref{eH_mixed}--\eqref{eH_mixed_3} are satisfied.
The terms arising in Proposition \ref{prop:7.1}
can be calculated as follows. First,
\begin{align*}
    T_1^2&=t^{3}\int \BE\big[H_x(\eta_t)
    H_y(\eta_t)\big]|x\wedge y|\,\lambda^2(d(x,y))\\
    &= t^{3} \int e^{-t(|x|+|y|-|x\wedge y|)}
    |x\wedge y|\, \lambda^2(d(x,y)),
  \end{align*}
  where $x\wedge y$ denotes the coordinatewise minimum of
  $x,y\in[0,1]^d$. Fix a (possibly empty) set
  $I\subseteq \{1,\dots,d\}$, let $J:=I^c$, and denote by $x^I$ and
  $x^J$ the subvectors of $x\in[0,1]^d$ formed by coordinates from $I$
  and $J$. It suffices to restrict the integration domain to the set
  where $x\wedge y=(x^I,y^J)$ and let $T^2_{1,I}$ be the corresponding integral. Let $m$ denote the cardinality of $I$.
  If $m=0$, then
  \begin{align*}
    T_{1,I}^2 =  t^{3} \int e^{-t|x|} |y|\I\{y\prec x\}\,\lambda^2(d(x,y))
    = 2^{-d}t^{3} \int
    e^{-t|x|} |x|^2\,\lambda(dx) \leq 27\cdot 2^{-d}\sigma_{t/3}^2.
  \end{align*}
  Here and in what follows we use the inequality $se^{-s}\leq 1$
  with $s=t|y|$, which yields that
  \begin{align*}
    t^i\int |y|^{i-1} e^{-t|y|} \,\lambda(dy)
    \leq t\int (t|y|e^{-t|y|/i})^{i-1} e^{-t|y|/i}\,\lambda(dy)
    \leq i^{i}\sigma_{t/i}^2,\quad i\in\N.
  \end{align*}
  The same calculation applies if $m=d$. If $m\in\{1,\dots,d-1\}$,
  then
  \begin{multline*}
    T_{1,I}^2=
    t^{3} \int_{[0,1]^d}e^{t|x^I|\,|y^J|}|x^I|\,|y^J|
    \left(\int_{[0,1]^{m}}e^{-t|y^I|\,|y^J|}\I\{x^I\prec y^I\}\, dy^I\right)\\
    \times\left(\int_{[0,1]^{d-m}} e^{-t|x^I|\,|x^J|}\I\{y^J\prec x^J\}\,dx^J\right)
    \,\lambda(d(x^I,y^J)). 
  \end{multline*}
  It can be shown by a small adaptation of the proof of
  \cite[Lemma~3.1]{bhat:mol21}, that
  \begin{displaymath}
    s\int_{[0,1]^m} e^{-s|x|}\I\{y\prec x\} \, dx
    \leq Ce^{-s|y|/a}\Big[1+\big|\log(s|y|)\big|^{m-1}\Big], \quad
    y\in[0,1]^m, 
  \end{displaymath}
  for any $a>1$ and a constant $C$ that depends on $m$ and $a$.  Let
  $a\in(1,2)$. Then, with $s:=t|y^J|$, we have
  \begin{displaymath}
    t\int_{[0,1]^m} e^{-t|y^I|\,|y^J|} \I\{x^I\prec y^I\} |y^J| dy^I
    \leq Ce^{-t|y^J||x^I|/a}\Big[1+\big|\log(t|y^J||x^I|)\big|^{m-1}\Big]. 
  \end{displaymath}
  By applying the same argument to the integral over $[0,1]^{d-m}$, we
  have that
  \begin{align*}
    T_{1,I}^2 \leq C^2 t \int
    e^{-t|z|(2/a-1)}\Big[1+\big|\log(t|z|)\big|^{m-1}\Big]
    \Big[1+\big|\log(t|z|)\big|^{d-m-1}\Big]\,\lambda(dz). 
  \end{align*}
  This is of the order $\mathcal{O}(\log^{d-1} t)$ by considering all
  summands separately and following the proof of
  \cite[Lemma~3.2]{bhat:mol21}.

  In this setting, $h_i(y)=t|y|$ for all $i$ and $\tilde{h}(y)=t|y|$.
  Further terms can be calculated as follows:
  \begin{align*}
    T_3&=t\int e^{-t|x|}\,\lambda(dx)=\sigma_t^2,\\
    T_4&\leq 5t^2\int |y|e^{-t|y|}\,\lambda(dy)\leq 20\sigma_{t/2}^2,\\
    T_5&\leq 8 t^3\int |y|^2e^{-t|y|}\,\lambda(dy)\leq 216 \sigma_{t/3}^2,
  \end{align*}
  and the terms involved in the bound on the Kolmogorov distance are
  \begin{align*}
    T_6&= \Big(\int t^3|y|^2(1+t|y|) e^{-t|y|} \,\lambda(dy)\Big)^{1/2}
         \leq (27\sigma_{t/3}^2+256\sigma_{t/4}^2)^{1/2},\\
    T_7&= \Big(t\int e^{-t|x|}\,\lambda(dx)\Big)^{1/2}=\sigma_t,\\
    T_9&= \Big(\int \big[3t^2+3|y|t^3+2|y|^2t^4\big]|y|
         e^{-t|y|}\,\lambda(dy)\Big)^{1/2}
         \leq (12\sigma_{t/2}^2+81\sigma_{t/3}^2+512\sigma_{t/4}^2)^{1/2}.
  \end{align*}
  Noticing that $\sigma_t^2=\BV \deltaKS(G)$ behaves like $\log^{d-1} t$, we obtain from Proposition \ref{prop:7.1} that
  \begin{align*}
    \max\Big(d_W(\sigma_t^{-1}\deltaKS(G),N),
    d_K(\sigma_t^{-1}\deltaKS(G),N)\Big)
    =\mathcal{O}(\sigma_t^{-1}). 
  \end{align*}
\end{example}

\section*{Acknowledgement}
\label{sec:acknowledgements}

IM and MS have been supported by the Swiss National Science Foundation
Grant No. 200021\_175584. The authors are grateful to two
  referees for several stimulating comments which led to an
  improvement of the presentation.

\end{document}